\documentclass[a4paper,twoside,12pt]{article}
\usepackage[english]{babel} %francais, polish, spanish, ...
\usepackage[T1]{fontenc}
\usepackage[latin1]{inputenc}
\usepackage{cmbright}

\usepackage{lmodern} %Type1-font for non-english texts and characters

\usepackage{graphicx} %%For loading graphic files
\usepackage{multirow}
\usepackage{tikz} %%Generate vector graphics from within LaTeX
\usetikzlibrary{calc}
\usetikzlibrary{arrows}
\usetikzlibrary{decorations.pathreplacing}  %for underbrace in tikz
\usepgflibrary{shapes.geometric}

\usepackage{ifthen}

\usepackage{amsmath}
\usepackage{amsthm,thmtools}
\usepackage{amsfonts}
\usepackage{amssymb,latexsym}

\usepackage{mathtools}

\usepackage{a4wide} %%Smaller margins = more text per page.

\usepackage[scale=1.0]{ccicons}

\usepackage{color}
\usepackage{url}
\usepackage{cite}
\usepackage{booktabs}

\usepackage
[colorlinks,linkcolor=black,citecolor=black,filecolor=black,urlcolor=black,unicode]
{hyperref}
\definecolor{darkgreen}{rgb}{0,0.4,0}
\definecolor{BrickRed}{rgb}{0.65,0.08,0}
\hypersetup{colorlinks=true,linkcolor=blue,citecolor=red,filecolor=BrickRed,urlcolor=darkgreen}

\newcommand{\stepset}{\mathcal{S}}

\newcommand{\walksym}{\omega}
\newcommand{\walk}[1]{\walksym_{#1}}

\newcommand{\neigh}{\Omega}
\newcommand{\neighray}[1]{\neigh \setminus #1}
\newcommand{\neighrayrho}{\neighray{(\rho,\infty)}}

\newcommand{\E}{\mathbb{E}}

\newcommand{\N}{\mathbb{N}}
\newcommand{\PR}{\mathbb{P}}

\newcommand{\R}{\mathbb{R}}

\newcommand{\V}{\mathbb{V}}
\newcommand{\Z}{\mathbb{Z}}

\newcommand{\Hc}{\mathcal{H}}
\newcommand{\Nc}{\mathcal{N}}
\newcommand{\Rc}{\mathcal{R}}

\newcommand{\LandauO}{\mathcal{O}}
\newcommand{\Landauo}{o}

\newcommand{\diamondend}
{ \hfill ${\lozenge} $}

\theoremstyle{plain}% default
\newtheorem{theorem}{Theorem}[section]
\newtheorem{lemma}[theorem]{Lemma}
\newtheorem{proposition}[theorem]{Proposition}

\newtheorem{example}[theorem]{Example}

\theoremstyle{definition}
\newenvironment{definition}[1][]{\refstepcounter{theorem} 
	\medskip \noindent \textbf{\textit{Definition~\thetheorem}}%
	\ifthenelse{\equal{#1}{}}{.}{~(#1).}%
}{ \diamondend \medskip}

\theoremstyle{remark}
\newtheorem{remark}[theorem]{Remark}

\newenvironment{hypo}[1][]{\medskip \noindent %
\textbf{\textit{Hypothesis #1:}} }
{ \diamondend \medskip }

\newcommand{\roleW}{r\^ole}
\newcommand{\role}{\roleW~}

\usepackage{fancyhdr}
\usepackage{titling}
\fancypagestyle{firststyle}{%
  \fancyhf{}%
  \fancyfoot[L]{\small \ccbyncnd~2020. This manuscript version is made available under the CC-BY-NC-ND 4.0 license \url{http://creativecommons.org/licenses/by-nc-nd/4.0/}}%
}

\begin{document}

%%%%%%%%%%%%%%%%%%%%%%%%%%%%
%% Title
%%%%%%%%%%%%%%%%%%%%%%%%%%%%

\author{{Michael Wallner}$^1$\\[1cm]
	 {$^1$} Institute of Discrete Mathematics and Geometry, TU Wien, Austria \\
	 %Wiedner Hauptstr. 8-10/104, A-1040 Wien \\
	LaBRI, 
	Universit\'e de Bordeaux,
	France \\
	\url{http://dmg.tuwien.ac.at/mwallner/}}

\title{A half-normal distribution scheme for\\ generating functions 
}
\date{}

\maketitle
\thispagestyle{firststyle}

{\it \small
This work supplants the extended abstract ``A half-normal distribution scheme for generating functions and the unexpected behavior of Motzkin paths'' which appeared in the Proceedings of the 27th International Conference on Probabilistic, Combinatorial and Asymptotic Methods for the Analysis of Algorithms (AofA 2016) Krakow Conference.}

\begin{abstract}
We present a general theorem on the structure of bivariate generating functions which gives sufficient conditions such that the limiting probability distribution is a half-normal distribution. 
If $X$ is a normally distributed random variable with zero mean, then $|X|$ obeys a half-normal distribution.
In the second part, we apply our result to prove three natural appearances in the domain of lattice paths: the number of returns to zero, the height, and the sign changes are under zero drift distributed according to a half-normal distribution.
This extends known results to a general step set.
Finally, our result also gives a new proof of Banach's matchbox problem.
\end{abstract}

%%%%%%%%%%%%%%%%%%%%%%%%%%%%%%%%%%%%%%%%%%%%%%%%%%%%%%%%%%%%%%%%%%%%%%%
%% main text

%%%%%%%%%%%%%%%%%%%%%%%%%%%%%%%%
%% Introduction
%%%%%%%%%%%%%%%%%%%%%%%%%%%%%%%%
\section{Introduction}
\label{sec:intro}

Generating functions $\sum f_{n} z^n$ are a powerful tool in combinatorics and probability theory. One of the main reasons of their success is the symbolic method~\cite{flaj09}, a general correspondence between combinatorial constructions and functional equations. It provides a direct translation of the structural description of a class into an equation on generating functions without the necessity of first deriving recurrence relations on the level of counting sequences. 
A natural generalization are bivariate generating functions $\sum_{n,k} f_{nk} z^n u^k$.
We say the variable $u$ ``marks'' a parameter, like e.g.~the height, if $\sum_{k} f_{nk} = f_n$ and $f_{nk} \geq 0$. 
In general,~$n$ is the length or size, and~$k$ is the value of a ``marked'' parameter. 

Analytic schemes for generating functions, or short only schemes, are general methods to deduce properties of the counting sequences $f_{nk}$ by solely analyzing the properties of  their bivariate generating functions. 
Like a black box, after checking algebraic, analytic, and geometric properties, the governing limit distribution follows directly.
Many different distributions are covered in the literature and most of them were celebrated results with many applications.
Let us mention some of them next.

The most ``normal'' case is a Gaussian limit distribution. Many schemes describe necessary conditions for its existence: Hwang's quasi-powers theorem~\cite{hwan98}, the supercritical composition scheme~\cite[Proposition~IX.6]{flaj09}, the algebraic singularity scheme \cite[Theorem~IX.12]{flaj09}, a square-root singularity scheme~\cite{DrSo97}, an implicit function scheme for algebraic singularities \cite[Theorem~2.23]{drmo09}, and the limit law version of the Drmota--Lalley--Woods Theorem~\cite[Theorem~8]{badr15}. 
But such schemes also exist for other distributions such as for example the Airy distribution~\cite{bfss01}, or the Rayleigh distribution~\cite{DrSo97}. 
In general, it was shown in~\cite{bbpb12} and \cite[Theorem~10]{badr15} that even in simple examples ``any limit law'' is possible (the limit law can be arbitrarily close to any c\`adl\`ag multi-valued curve in $[0,1]^2$). 
Let us also mention the recent study~\cite{hwang2018asymptotic}, which compares hundreds of combinatorial sequences and their limit distributions to find common features using the method of moments and also treats the half-normal distribution.

The main result of this paper is the first scheme stating sufficient conditions for a half-normal distribution given in Theorem~\ref{theo:theo4}.
This distribution is generated by the absolute value $|X|$ of a normally distributed random variable $X$ with mean zero. 
We compare the key properties of some relevant distributions in Table~\ref{tab:compprobdis}. 
Our result falls into the class of generating functions analyzed by Drmota and Soria in~\cite{DrSo95,DrSo97}, as a similar structure of a square-root type is present.

\linespread{1.07}
\begin{table}[h!]
	\centering
	\begin{tabular}{cccc}
		\toprule
		&
		\underline{\bf Normal}
		&
		\underline{\bf Half-normal}
		&
		\underline{\bf Rayleigh}
		\\
		&
		$
			\Nc(\mu,\sigma^2)
		$
		&
		$
			\Hc(\sigma)
		$
		&
		$
			\Rc(\sigma)
		$
		\\
		%%%%		
		%%%%		
		\midrule  &
		{\includegraphics[width=0.17\textwidth]{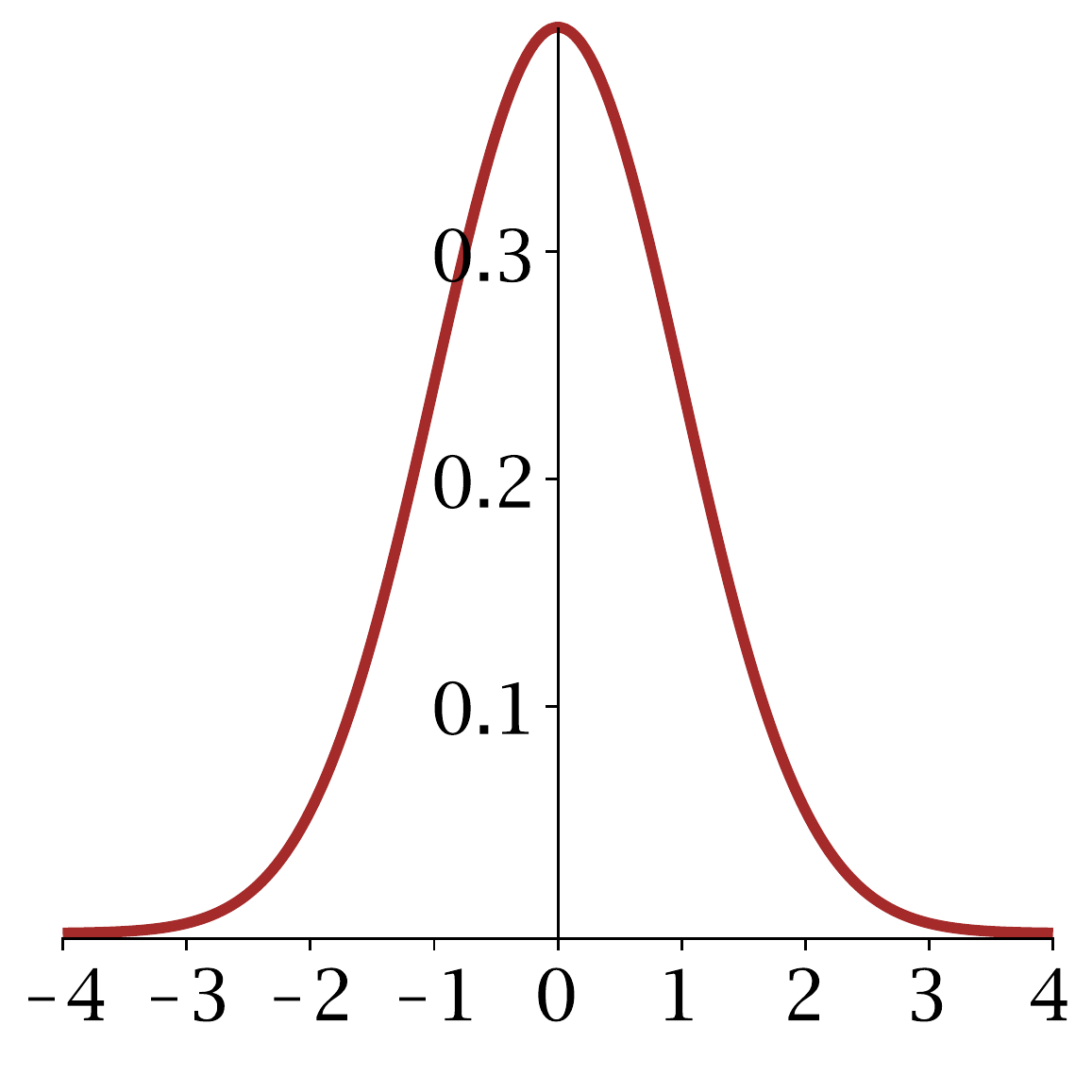}}
		&
		{\includegraphics[width=0.17\textwidth]{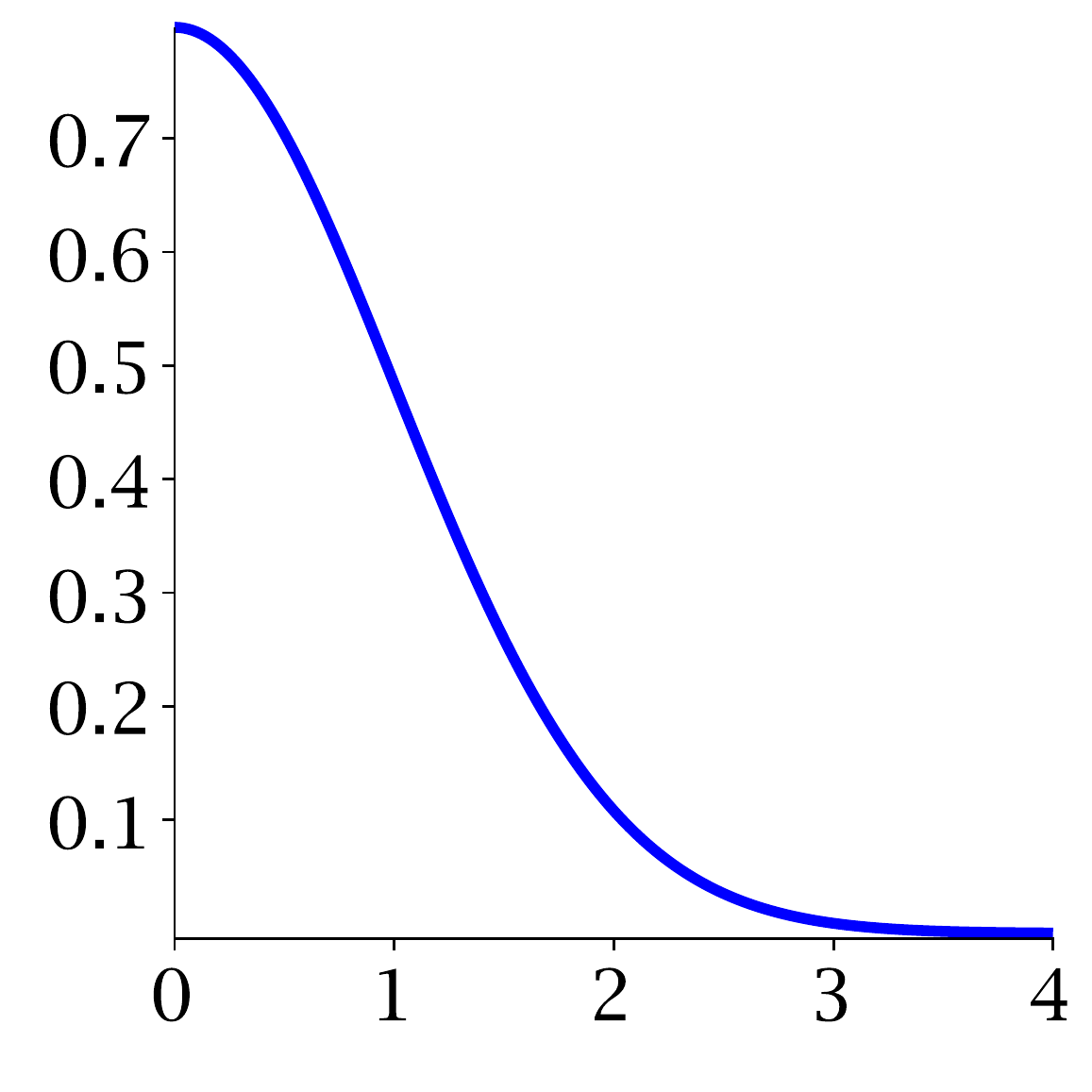}}
		&
		{\includegraphics[width=0.17\textwidth]{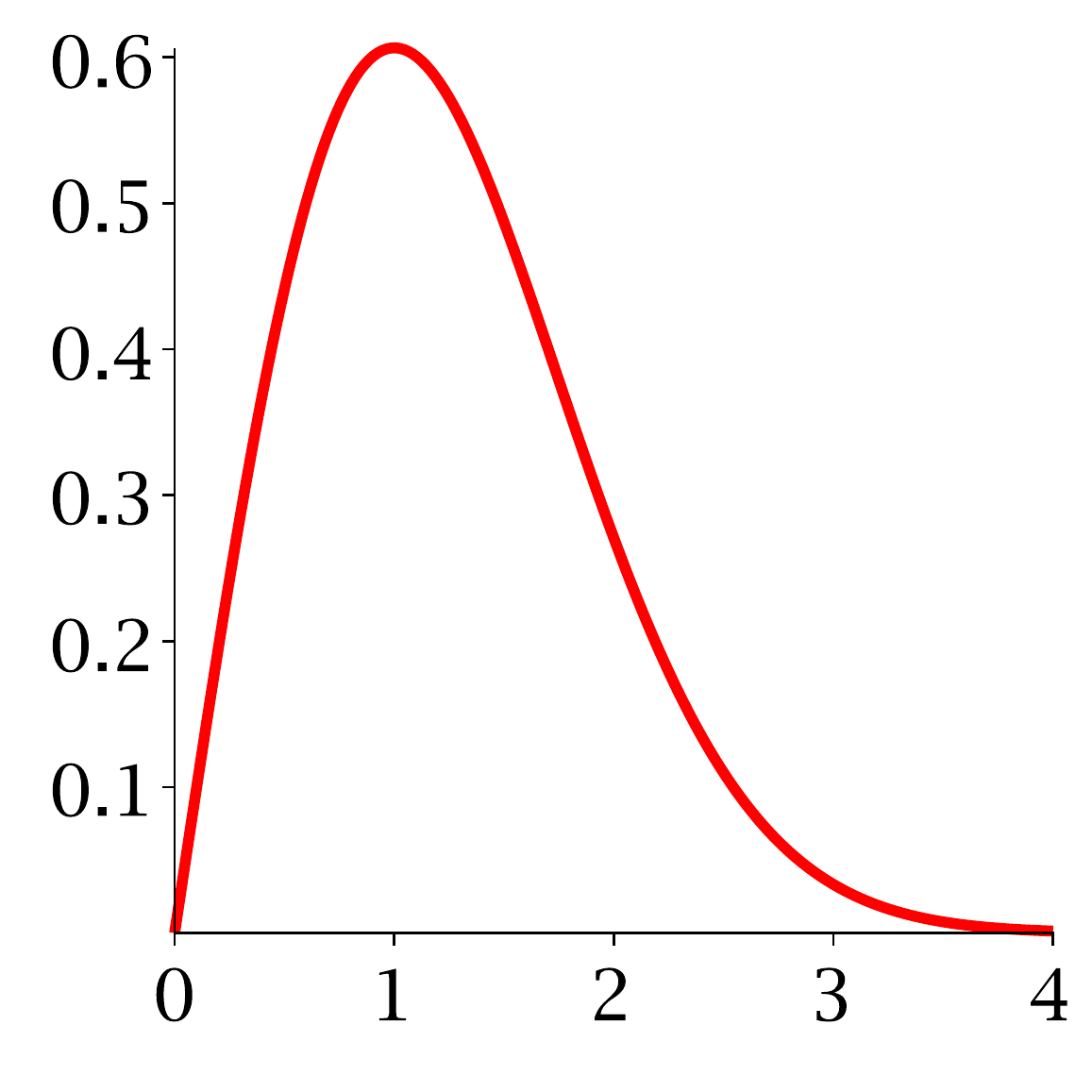}}
		\\[-1mm]
		%%%%
		%\hline
		Support
		&
		$
			x \in \R
		$
		&
		$
			x \in \R_{\geq 0}
		$
		&
		$
			x \in \R_{\geq 0}
		$
		\\
		%%%%
		%\hline
		PDF
		&
		$
			\frac{1}{\sqrt{2 \pi \sigma^2}} \exp\left(-\frac{(x-\mu)^2}{2\sigma^2}\right)
		$
		&
		$
			\sqrt{\frac{2}{\pi \sigma^2}} \exp\left(-\frac{x^2}{2\sigma^2}\right)
		$
		&
		$
			\frac{x}{\sigma^2} \exp\left(-\frac{x^2}{2\sigma^2}\right)
		$
		\\
		%%%%
		%\hline
		Mean
		&
		$
			\mu
		$
		&
		$
			\sigma\sqrt{\frac{2}{\pi}}
		$%
		&
		$
			\sigma\sqrt{\frac{\pi}{2}}
		$%
		\\
		%%%%
		%\hline
		Variance
		&
		$
			\sigma^2
		$
		&
		$
			\sigma^2 \left(1 - \frac{2}{\pi}\right)
		$
		&
		$
			\sigma^2 \left(2 - \frac{\pi}{2} \right)
		$
		\\
		\bottomrule
	\end{tabular}
	\caption{A comparison of the normal, half-normal, and Rayleigh distribution. 
	We will encounter all of them in the context of lattice paths.
	}
	\label{tab:compprobdis}
\end{table}
\linespread{1.0}

\vspace{-0.2\baselineskip}

Despite its rare emergence in existing literature we also present four natural appearances of the half-normal distribution in lattice path theory: the number of returns to zero, the number of sign changes, the height of unconstrained walks, and the final altitude in the reflection model~\cite{bawa15c} (which encodes Banach's matchbox problem).
Our Theorem~\ref{theo:theo4} yields these results in a unified manner and pinpoints their common nature: a special square-root singularity, which is in all theses cases due to a coalescence of singularities.
These results extend previous ones on random walks and other special steps~\cite[Chapter III]{fell68} to arbitrary aperiodic lattice paths and show that the same phenomena prevail.
These include a phase transition in the growth when the drift changes from non-zero to zero and a change in the nature of the law; see Table~\ref{tab:compretsign}. 
Notably, the expected value for $\Theta(n)$ trials grows like $\Theta(\sqrt{n})$ and not, as might be expected, linearly.

The common feature of all examples is that after the combinatorial problem has been translated into generating functions, Theorem~\ref{theo:theo4} can be used as a black box to prove the appearance of a half-normal distribution.
The necessary generating function relations for the first three cases are summarized in Table~\ref{tab:BGFrelations}.
They were discussed in the context of Motzkin walks in the extended abstract \cite{wall16} (see Table~\ref{tab:compretsign}).

\begin{table}[ht!]
	\begin{center}
\linespread{1.5}
	\begin{tabular}{c@{\hspace*{10mm}}ccc}
		\toprule
		Marked Parameter      & BGF & Equation \\
		\midrule  
			Returns to zero & $\frac{W(z)}{u + (1-u) B(z)}$ & \eqref{eq:returnsBGF} \\
			Height & $\frac{W(z)M(z,u)}{M(z)} 
			%= - \frac{1}{p_d z} \left(  \prod_{j=1}^c \frac{1}{1-u_{j}(z)} \right) \left( \prod_{\ell=1}^d \frac{1}{u-v_{\ell}(z)}\right)
			$ & \eqref{eq:Fzuheight}\\
			Sign changes (Motzkin) &  $B(z,u) \frac{W(z)}{B(z)} + \frac{B(z,u)-C(z)}{2} \left(\frac{W(z)}{B(z)} - 1 \right) (u-1)$ & \eqref{eq:Fzusignchanges}
			\\
		\bottomrule
  \end{tabular}
\end{center}
\linespread{1.0}
\caption{
Relations for the bivariate generating functions (BGF) $W(z,u) = \sum_{n,k} w_{nk}z^k u^k$ of walks with marked parameter. The functions $W(z), B(z),M(z)$ are the GFs of walks, bridges, meanders, respectively; see Section~\ref{sec:typesofpaths}. $M(z,u)$ is the GF of meanders of length $n$ with marked final altitude, $B(z,u)$ is the GF of bridges of length $n$ with marked sign changes, $C(z) = \frac{1}{1-p_0 z}$. 
}
\label{tab:BGFrelations}
\end{table}

{\bf Plan of this article.} 
In Section~\ref{sec:halfnormal}, we present our main contribution: a scheme for bivariate generating functions leading to a half-normal distribution. 
In Section~\ref{sec:motzprop}, we apply our result to lattice paths. 
First, we give a brief introduction to the theory of analytic lattice path counting in Section~\ref{sec:typesofpaths}.
Then we apply our result to the number of returns to zero in Section~\ref{sec:returns}, the height in Section~\ref{sec:height}, the number of sign changes in Section~\ref{sec:sign} (where sign changes are treated only in the case of Motzkin walks), and to Banach's matchbox problem in Section~\ref{sec:Banach}. In the case of a zero drift a half-normal distribution appears in all cases.

%%%%%%%%%%%%%%%%%%%%%%%%%
%% General theorems
%%%%%%%%%%%%%%%%%%%%%%%%%
\section{The half-normal theorem}
\label{sec:halfnormal}

Let $f(z) = \sum_{n \geq 0} f_n z^n$ with $f_n>0$ be a generating function and $f(z,u) = \sum_{n,k \geq 0} f_{nk} z^n u^k$ with $f_{nk}\geq0$ be the corresponding bivariate generating function where a parameter has been marked such that $f(z,1) = f(z)$. 
For fixed $n \in \N$ the numbers $f_{nk}$ implicitly define a (discrete) probability distribution. In particular, we define a sequence of random variables $X_n, n \geq 1$, by
\begin{align*}
	\PR(X_n = k) &:= \frac{f_{nk}}{f_n} = \frac{[z^n u^k] f(z,u)}{[z^n] f(z,1)}.
\end{align*}
Our goal is to identify the limit distribution of these random variables.
This is achieved by a careful analysis of algebraic and analytic properties of $f(z,u)$. 

The half-normal theorem requires certain technical conditions to be satisfied which are summarized in the following Hypothesis~[H'].
Note that if one additionally assumes $h(\rho,1) >0$ and $\rho=\rho(u)$ being a (possibly non-constant) function of $u$, as well as $1/f(z,u)$ can be analytically continued, then it is equivalent to~\cite[{Hypothesis~{[H]}}]{DrSo97}. Nevertheless, we fully state it here for the convenience of the reader.

\begin{hypo}[{[H']}]
	Let $f(z,u) = \sum_{n,k} f_{nk} z^n u^k$ be a power series in two variables with non-negative coefficients $f_{nk} \geq 0$ such that $f(z,1)$ has a radius of convergence $\rho > 0$. 
	
	We suppose that $1/f(z,u)$ has the local representation
	\begin{align}
		\label{eq:decompFinv}
		\frac{1}{f(z,u)} &= g(z,u) + h(z,u) \sqrt{1 - \frac{z}{\rho}},
	\end{align}
	for $|u-1| < \varepsilon$ and $|z - \rho| < \varepsilon$, $\arg(z-\rho) \neq 0$, where $\varepsilon > 0$ is some fixed real number, and $g(z,u)$ and $h(z,u)$ are analytic functions.
	Furthermore, we have $g(\rho,1)=0$.
	
	In addition, $z=\rho$ is the only singularity on the circle of convergence $|z| = \rho$, and 
	$f(z,u)$ can be analytically continued to a region $|z| < \rho + \delta, |u|<1+\delta, |u-1| > \frac{\varepsilon}{2}$, and $\arg(z-\rho) \neq 0$ for some $\delta > 0$. 
\end{hypo}

\begin{theorem}[Half-normal limit theorem]
	\label{theo:theo4}
	Let $f(z,u)$ be a bivariate generating function satisfying [H']. If $g_z(\rho,1) \neq 0$, $h_u(\rho,1) \neq 0$, and $h(\rho,1) = g_u(\rho,1) = g_{uu}(\rho,1) = 0$, then the sequence of random variables $X_n$ defined by
	$
		\PR(X_n = k) = \frac{[z^n u^k] f(z,u)}{[z^n] f(z,1)},
	$
	has a half-normal limiting distribution, i.e.,
	\begin{align*}
		\frac{X_n}{\sqrt{n}} \stackrel{\text{d}}{\to} \Hc(\sigma),
	\end{align*}
	where $\sigma = \sqrt{2}\frac{|h_u(\rho,1)|}{\rho |g_z(\rho,1)|}$, and $\Hc(\sigma)$ has density $\frac{\sqrt{2}}{\sqrt{\pi \sigma^2 }} \exp\left( - \frac{z^2}{2 \sigma^2} \right)$ for $z \geq 0$. Expected value and variance are given by
	\begin{align*}
		&& 
		\E[X_n] &= \sigma \sqrt{\frac{2}{\pi}} \sqrt{n} + \LandauO(1) & \text{and} &&
		\V[X_n] &= \sigma^2 \left( 1 - \frac{2}{\pi} \right) n + \LandauO(\sqrt{n}). 
		&&
	\end{align*}
	
	Moreover, we have the local law
	\begin{align*}
		\PR(X_n = k) &= \frac{1}{\sigma} \sqrt{\frac{2}{\pi n}} \exp\left( -\frac{k^2/n}{2\sigma^2 } \right) + \LandauO\left(kn^{-3/2}\right) + \LandauO\left( n^{-1} \right),
	\end{align*}
	uniformly for all $0 \leq k \leq k_0 \sqrt{n \log n}$ with $k_0 < \sigma$.
\end{theorem}

In the proof of Theorem~\ref{theo:theo4} we will need the following representation of the characteristic function of the half-normal distribution; see~\cite[Equation~($15$)]{tbha14}. 
The following lemma is analogous to \cite[Lemma~6%
% and 7
]{DrSo97} where an integral representation of the characteristic function of the Rayleigh distribution is given. We omit the proof as it follows the same lines. 

\begin{lemma}
	\label{lem:charHalfnormal}
	Let $\varepsilon>0$ and let $\gamma$ be the Hankel contour starting from $-\varepsilon i + \infty$, passing around $0$ and tending to $\varepsilon i +\infty$. Then
	\begin{align*}
		\frac{1}{2 \pi i} \int_{\gamma} \frac{e^{-s}}{s+ix\sqrt{-s}} \, ds &= \varphi_{\Hc}\left(\sqrt{2}x\right), 
	\end{align*}	
	where $\varphi_{\Hc}(t) = \sqrt{\frac{2}{\pi}} \int_{0}^{\infty} e^{its} e^{-s^2/2} \, ds$ is the characteristic function of the half-normal distribution with $\sigma=1$.
\end{lemma}

\begin{proof}[Proof of Theorem~\ref{theo:theo4}]
	The proof follows the same steps as the one of~\cite[{Theorem 1}]{DrSo97}. 
	Therefore, we restrict ourselves to only highlight the main differences. 

	First, we derive asymptotic expansions for mean and variance using singularity analysis, which transfers a singular expansion of the form $(1-z/\rho)^{\alpha}$ in the expansion of $f(z)$ close to a singularity $\rho$ into the corresponding asymptotic term $\rho^{-n} n^{-\alpha-1}/\Gamma(-\alpha)$ of the coefficients $[z^n]f(z)$ for $n \to \infty$; see~\cite{FlajoletOdlyzko1990SingularityAnalysis,flaj09}.
	Due to $g(\rho,1) = h(\rho,1) = 0$, and $g_z(\rho,1) \neq 0$ we get from~\eqref{eq:decompFinv} that
	\begin{align}
		[z^n] f(z,1) 
		             &= -\frac{\rho^{-n}}{\rho g_z(\rho,1)} \left( 1 + \LandauO(n^{-1/2}) \right). \label{eq:theo4cAsym}
	\end{align}
	Analogously, because of $h_u(\rho,1) \neq 0$, and $h(\rho,1) = g_u(\rho,1) = g_{uu}(\rho,1) = 0$ we get
	\begin{align*}
		[z^n] f_u(z,1) &= -\frac{2 h_u(\rho,1) \rho^{-n}}{(\rho g_z(\rho,1))^2} \sqrt{\frac{n}{\pi}} \left( 1 + \LandauO(n^{-1/2}) \right), \\ {} 
		[z^n] f_{uu}(z,1) &= - \frac{2 h_u(\rho,1)^2}{(\rho g_z(\rho,1))^3} \rho^{-n} n \left(1 + \LandauO(n^{-1/2})\right).
	\end{align*}
	Hence,
	\begin{align*}
		\E[X_n] &= \frac{[z^n] f_u(z,1)}{[z^n]f(z,1)} = \frac{2 h_u(\rho,1)}{\rho g_z(\rho,1)} \sqrt{\frac{n}{\pi}} \left( 1 + \LandauO(n^{-1/2}) \right),\\
		\V[X_n] &= \frac{[z^n] f_{uu}(z,1)}{[z^n]f(z,1)} + \E[X_n] - \E[X_n]^2 = 2 \left( \frac{h_u(\rho,1)}{\rho g_z(\rho,1)} \right)^2 \left(1 - \frac{2}{\pi}\right) n + \LandauO(n^{1/2}).
	\end{align*}
	These results strongly suggest that the underlying limit distribution is a half-normal one. 
	We continue by deriving the asymptotic form of the characteristic function of $X_n/\sqrt{n}$. 
	
	Note that the same contour of integration as in~\cite{DrSo97} sketched in Figure~\ref{fig:hankel} can be used. 
	Therefore, we need the following expansions coming from the substitutions $z = \rho\left(1 + \frac{s}{n} \right)$ and $u = e^{it/\sqrt{n}} = 1 + \frac{it}{\sqrt{n}} + \LandauO(n^{-1})$:
	\begin{align}
		\label{eq:ghexpanded}
		\begin{split}
		g\left(\rho\left(1 + \frac{s}{n} \right), e^{it/\sqrt{n}}\right) &= g_z(\rho,1) \rho\frac{s}{n} + \LandauO\left(\frac{|s|}{n^{3/2}}\right), \\
		h\left(\rho\left(1 + \frac{s}{n} \right), e^{it/\sqrt{n}}\right) &= h_u(\rho,1) \frac{it}{\sqrt{n}} + \LandauO\left(\frac{|s|}{n}\right),
		\end{split}
	\end{align}
	as $g_u(\rho,1)=0$ and $h(\rho,1)=0$. We want to emphasize that this behavior is different from the one in \cite[{Theorem 1}]{DrSo97}, wherein $g_u(\rho,1)=1$ and $h(\rho,1)>0$ is required.  
	Thus, we get for the Cauchy integral along the contour $\Gamma_1$ (see Figure~\ref{fig:hankel})
	\begin{align}
		\frac{1}{2 \pi i} \int_{\Gamma_1} f(z,u) \frac{dz}{z^{n+1}} 
			&= \frac{\rho^{-n}}{2 \pi i} \int_{\gamma'} \frac{e^{-s}\left(1 + \LandauO\left(\frac{|s|}{n}\right) \right)}{g_z(\rho,1) \rho\frac{s}{n} + h_u(\rho,1) it\frac{\sqrt{-s}}{n} + \LandauO\left(\frac{|s|}{n^{3/2}}\right)} \, \frac{ds}{n} \notag \\
			&= \frac{\rho^{-n}}{ \rho g_z(\rho,1)}\frac{1}{2 \pi i} \int_{\gamma'} \frac{e^{-s}}{s + \sqrt{-s} i \frac{h_u(\rho,1) t}{ \rho g_z(\rho,1)} } \, ds + \LandauO\left(\frac{\rho^{-n}}{n^{1/2}}\right). \label{eq:charHalfnormala}
	\end{align}
	
	The other computations are again analogous to~\cite{DrSo97}. 
	By~\eqref{eq:charHalfnormala} and Lemma~\ref{lem:charHalfnormal} we get 
	\begin{align}
		\label{eq:theo4Gamma1res}
		\frac{1}{2 \pi i} \int_{\Gamma_1} f(z,u) \frac{dz}{z^{n+1}} 
			&= \frac{\rho^{-n}}{ \rho g_z(\rho,1)} \varphi_{\Hc} \left( \frac{\sqrt{2}h_u(\rho,1)}{\rho g_z(\rho,1)} t \right) + \LandauO\left(\frac{\rho^{-n}}{n^{1/2}}\right). 
	\end{align}

	\begin{figure}%[ht]
		\begin{center}	
			\includegraphics[width=0.25\textwidth]{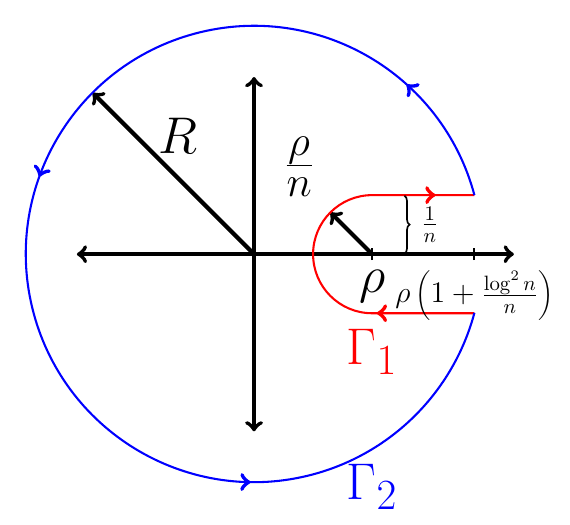} \qquad \qquad \qquad
			\includegraphics[width=0.25\textwidth]{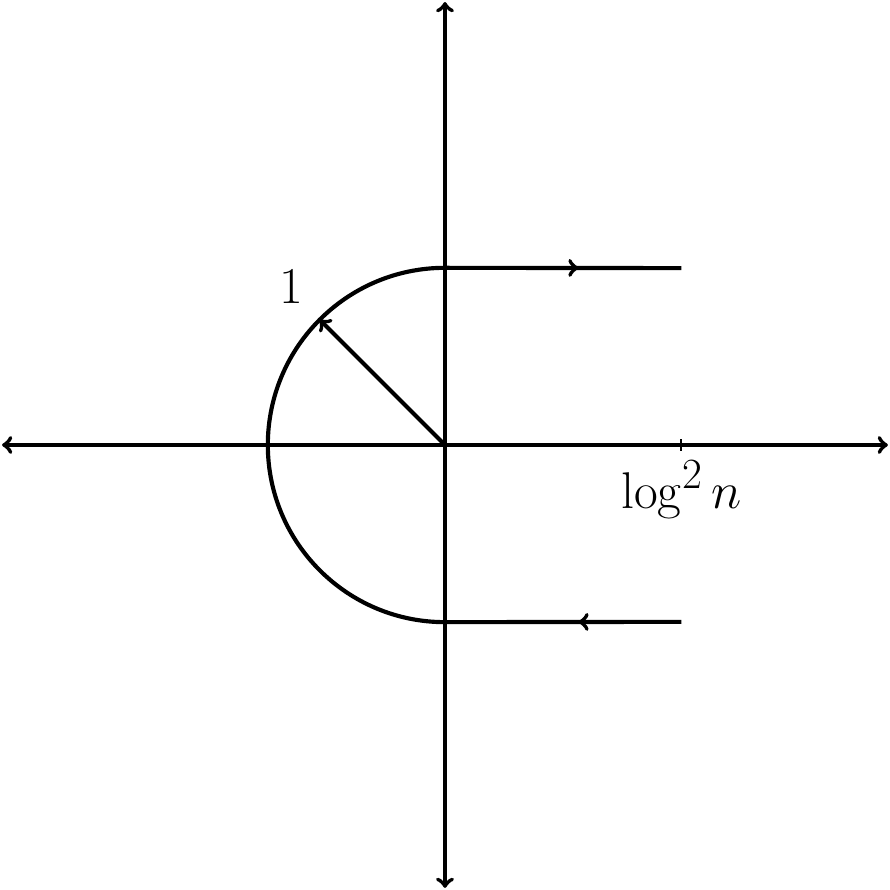}
			\caption{Hankel contour decomposition of $\Gamma$ (left), and contour of $\gamma'$ (right).}
			\label{fig:hankel}
		\end{center}
	\end{figure}
	 
	What remains is to bound the remaining part of the integral. 
	Using the expansions from~\eqref{eq:ghexpanded} we directly get
	\begin{align*}
		f\left( \rho \left(1 + \frac{\log^2 n + i}{n}\right),e^{\frac{it}{\sqrt{n}}} \right) &= \LandauO\left( \frac{n}{\log^2 n} \right).
	\end{align*}
	The rest of the proof of the weak limit theorem follows by standard computations.
	
	In the proof of the local limit theorem we get a different polar singularity than in the Rayleigh case of the mapping $u \mapsto f(z,u)$. 
	For $z=\rho\left(1+\frac{s}{n}\right)$ and $u_0 = 1 + \frac{t_0}{\sqrt{n}}$ we get
	\begin{align*}
		t_0 &=  \frac{\rho g_z(\rho,1)}{h_u(\rho,1)} \sqrt{-s} + \LandauO\left( \frac{\sqrt{|s|}}{n} \right),
	\end{align*}
	with residue
	\begin{align*}
		\frac{1}{h_u(\rho,1)} \sqrt{\frac{n}{-s}} \left( 1 + \LandauO\left( \frac{\sqrt{|s|}}{n} \right) \right).
	\end{align*}
	Then the same steps as in the proof of~\cite{DrSo97} yield the result. The suitable range for $k$ follows from the stated error terms because the main term behaves for $k=k_0 \sqrt{n \log n}$ like $\LandauO(n^{-\frac{(k_0/\sigma)^2+1}{2}})$ and the second error term like $\LandauO(\frac{k_0 \sqrt{\log n}}{n})$.
\end{proof}

Let us compare this result with the one of the corresponding Rayleigh limit theorem~\cite[Theorem~1]{DrSo97} and give an intuitive explanation of the underlying idea.
They have most of their conditions in common, including the key assumption of a square-root decomposition~\eqref{eq:decompFinv}.
Indeed, their only difference is that the Taylor expansions of the functions $g(z,u)$ and $h(z,u)$ in~\eqref{eq:decompFinv} start in the half-normal case with more zero terms; compare Table~\ref{tab:differences}. This increases the order of the singularity which then changes the nature of the limit law.

\begin{table}[ht]
\begin{center}
\begin{tabular}{ccc}
	\toprule
	& Rayleigh & Half-normal  \\
	\midrule
$g(\rho,1)$ & $=0$ & $=0$ \\
$g_u(\rho,1)$ & $<0$ & $=0$ \\
$g_{uu}(\rho,1)$ & arbitrary & $=0$ \\
$g_z(\rho,1)$ & arbitrary & $\neq 0$ \\
\midrule
$h(\rho,1)$ & $>0$ & $=0$ \\
$h_u(\rho,1)$ & arbitrary & $\neq 0$ \\
	\bottomrule
\end{tabular}
\caption{Comparing conditions on $g(z,u)$ and $h(z,u)$ from the square-root decompositions~\eqref{eq:decompFinv} of the Rayleigh Theorem~\cite[Theorem~1]{DrSo97} and the half-normal Theorem~\ref{theo:theo4}.}
\label{tab:differences}
\end{center}
\end{table}

From a technical point of view, the main observation is that the contributions of $g(z,u)$ and $h(z,u)\sqrt{1-z/\rho}$ are in a suitable region $z \sim \rho$ and $u \sim 1$ of the same asymptotic order of magnitude.
In particular, the decomposition $g(z,u)+h(z,u)\sqrt{1-z/\rho}$ is after the substitutions $z = \rho\left(1 + \frac{s}{n} \right)$ and $u = e^{it/\sqrt{n}} = 1 + \frac{it}{\sqrt{n}} + \LandauO(n^{-1})$ of the order $(\sqrt{-s}+t)n^{-1/2}$ in the Rayleigh case and $(-s+t\sqrt{-s})n^{-1}$ in the half-normal case.
Observe that these two cases differ by a multiplicative factor $\sqrt{-s/n}$. 
In the proof above we have seen that this factor $\sqrt{-s}$ is responsible for the transition from a Rayleigh to a half-normal law; compare \cite[Lemma~6]{DrSo97} and Lemma~\ref{lem:charHalfnormal}.

%%%%%%%%%%%%%
\begin{example}
\label{ex:motzkin}
Consider unweighted unconstrained Motzkin walks with marked returns to zero, i.e.~points of altitude $y=0$. These are walks starting at $0$ composed of the steps $\{-1,0,1\}$.
In Equation~\eqref{eq:returnsBGF} we will see that their bivariate generating function is given by
\begin{align*}
	W(z,u) &= \frac{\sqrt{1+z}}{u \sqrt{1+z} (1-3z) + (1-u) \sqrt{1-3z}}.
\end{align*}
Due to $W(z,1) = \frac{1}{1-3z}$, we have $\rho=\frac{1}{3}$. The decomposition~\eqref{eq:decompFinv} is valid and we get
\begin{align}
	\label{eq:ghspecial}
	&&g(z,u) &= u(1-3z) && \text{ and } &
	h(z,u) &= \frac{1-u}{\sqrt{1+z}}.&&
\end{align}
All other conditions of Hypothesis~[H'] and Theorem~\ref{theo:theo4} are satisfied since $g(z,u)$ and $h(z,u)$ are analytic for $|z|<1$, and the analytic continuation beyond $z=1/3$ and $u=1$ trivially holds. 
This proves that the limit distribution of the number of returns to zero of unweighted Motzkin walks is a half-normal distribution. In Theorem~\ref{theo:mainRetZero} we will see that this holds for general lattice path models with zero drift.
\end{example}

\begin{example} 
	\label{ex:motzkingenb}
	Let us consider the following (artificial) generalization of Example~\ref{ex:motzkin}. 
	Define for $b \in [2/3,4/3]$ the bivariate generating function 
\begin{align*}
	W_b(z,u) &:= \frac{\sqrt{1+z}}{u \sqrt{1+z} (1-3z)^b + (1-u) \sqrt{1-3z}}.
\end{align*}
Note that it has non-negative coefficients and it specializes for $b=1$ to the generating function from Example~\ref{ex:motzkin}, i.e., $W_1(z,u)=W(z,u)$.
Hence, it makes sense to define a random variable $Y_n$ distributed according to the $n$-th coefficient of $z$. 
Then the same methods as presented in the proof of Theorem~\ref{theo:theo4} are applicable, yet giving a different limit distribution if $b \ne 1$.
\end{example} 

\begin{remark}
Let us here comment on another method to deduce the limit distribution: the method of moments; see e.g.~\cite{hwang2018asymptotic}. 
It proves convergence of moments which then implies weak convergence. 
For example, this method shows that the random variable $Y_n$ defined in Example~\ref{ex:motzkingenb} satisfies for $r >0$ and $n \to \infty$ 
\begin{align*}
	\E[Y_n^r] &\sim n^{(b-1/2)r} \left(\frac{\sqrt{3}}{2}\right)^r \frac{\Gamma(r+1) \Gamma(b)}{\Gamma\left( b + \frac{2b-1}{2}r \right)}.
\end{align*} 
The moments of the rescaled random variable $\tilde{Y}_n := \frac{2}{\sqrt{3}} \frac{Y_n}{n^{b-1/2}}$ satisfy Carleman's condition~\cite[pp.~189-220]{Carleman23}. Therefore, by the limit Theorem of Fr{\'e}chet and Shohat~\cite[p.~536]{frsh31}, $\tilde{Y}_n$ converges in distribution to a unique distribution given by the moments above. 
\end{remark}

\begin{remark}%[Non-constant singularity]
	\label{rem:strongermaintheo}
	The assumption of a constant singularity in $z$ given by $\rho$ can be weakened to a singularity $\rho(u) = \rho(1) + \LandauO((u-1)^3)$, i.e.,~$\rho'(1)=\rho''(1)=0$. However, to our knowledge no example is known where $\rho(u)$ is not constant in a neighborhood of $u \sim 1$.
	Additionally, this more general setting should lead to new limit theorems if $\rho'(1)\neq 0$ or $\rho''(1)\neq 0$. Yet again, to our knowledge no natural example is known. 
\end{remark}

Before we proceed, let us highlight the common (technical) link between our subsequent applications. In all four examples the decomposition~\eqref{eq:decompFinv} will have the generic structure
\begin{align*}
	\frac{1}{f(z,u)} &= \tilde{g}(z,u)\left(1-\frac{z}{\rho}\right) + \tilde{h}(z,u) (1-u) \sqrt{1 - \frac{z}{\rho}},
\end{align*}
with $\tilde{g}(\rho,1) \neq 0$, $\tilde{h}(\rho,1) \neq 0$, and $f(z,u)$ is analytically continuable in the necessary domain for Hypothesis~[H'] to hold; see e.g.~\eqref{eq:ghspecial}. This special form of $g(z,u)$ and $h(z,u)$ guarantees Theorem~\ref{theo:theo4} to hold and gives a half-normal distribution. Yet, the scheme does not need such a special factorization and holds in a more general setting.
Thus, it would be interesting if other ``natural'' appearances of such situations (and half-normal distributions in general) exist. 
So far we know of one such appearance in number theory~\cite{ggiv07}.

%%%%%%%%%%%%%%%%%%%%%%%%%%%%%
%% applications
%%%%%%%%%%%%%%%%%%%%%%%%%%%%%
\section{Applications to one dimensional lattice path counting}
\label{sec:motzprop}

In this section we show how Theorem~\ref{theo:theo4} solves many problems on aperiodic lattice paths and thereby generalizes many known results.
The necessary background on lattice paths is summarized in Section~\ref{sec:typesofpaths}, which is intended for 
readers unfamiliar with the exposition of Banderier and Flajolet \cite{bafl02} or related results.
The following examples are motivated by the very nice presentation of Feller \cite[Chapter~III]{fell68} on one-dimensional symmetric, simple random walks; see also~\cite[Chapter~12]{GrinsteadSnell2003Probability}. Therein, the discrete time stochastic process~$(S_n)_{n \geq 0}$ is defined by $S_0 = 0$ and $S_n = \sum_{j=1}^n X_j$, $n \geq 1$, where the $(X_i)_{i\geq 1}$ are i.i.d.~Bernoulli random variables with $\PR(X_i = 1) = \PR(X_i = -1) = \frac{1}{2}$. 
This directly corresponds to the step polynomial $P(u) = \frac{1}{2u} + \frac{u}{2}$. 
In particular compare \cite[Problems~$9$-$10$]{fell68} and \cite[Remark of Barton]{SkSh57} for returns to zero of symmetric and asymmetric random walks, respectively. Furthermore, we refer to \cite[Chapter~III.$5$]{fell68} for sign changes, and to \cite[Chapter~III.$7$]{fell68} for the height. 
Note that this area is still an active field of research with many different applications.
See for example \cite{dobl15} on an application of Stein's method on these parameters in which bounds for the convergence rate in the Kolmogorov and the Wasserstein metric are derived, 
\cite{BenNaimKrapivskyRandonFurling2016Maxima} where the maxima of two random walks are analyzed, 
and \cite{DevroyeLudosiNeu2013Prediction} for applications to machine learning.

\begin{table}[ht]
	\begin{center}
	\newcommand{\mygap}{2em} 
	\begin{tabular}{c@{\hspace*{3em}}
	                c@{\hspace*{\mygap}}
	                c@{\hspace*{\mygap}}
	                c}
		\toprule drift      & returns to zero & sign changes & height \\
		\midrule $\delta<0$ & 
			$\operatorname{Geom}\left( \frac{p_{-1}-p_{1}}{P(1)} \right)$ & 
			$\operatorname{Geom}\left( \frac{p_{1}}{p_{-1}} \right) $ & 
			$\operatorname{Geom}\left( \frac{p_{1}}{p_{-1}} \right) $
			\\
		%
		%\hline 
		$\delta=0$ & 
			$\Hc\left(\sqrt{\frac{P(1)}{P''(1)}}\right)$ & 
			$\Hc\left(\frac{1}{2} \sqrt{\frac{P''(1)}{P(1)}}\right) $ & 
			$\Hc\left(\sqrt{\frac{P''(1)}{P(1)}}\right) $
			\\
		%
		%\hline 
		$\delta>0$ & 
			$\operatorname{Geom}\left( \frac{p_{1}-p_{-1}}{P(1)} \right)$ & 
			$\operatorname{Geom}\left( \frac{p_{-1}}{p_{1}} \right)$ & 
			Normal distribution
			\\
		\bottomrule
  \end{tabular}
\end{center}
\caption{Limit laws for Motzkin walks with step polynomial $P(u) = p_{-1} u^{-1} + p_{0} + p_{1} u$ with drift $\delta = P'(1)$.
Let $X_n$ be the random variable of the respective parameter of Motzkin walks with $n$ steps. 
For the continuous distributions a rescaling is necessary: 
$\frac{X_n}{\sqrt{n}}$ for the half-normal distribution; $\frac{X_n - \mu n}{\sqrt{\sigma^2 n}}$ for the normal distribution. 
}
\label{tab:compretsign}
\end{table}

For the sake of brevity we will only mention the weak convergence law. However, in all cases the local law and the asymptotic expansions for mean and variance hold as well.
In Table~\ref{tab:compretsign} the different results of this section for the case of Motzkin paths are compared.
As can be seen, the result will depend on the sign of the drift.

\subsection{Lattice paths and the kernel method}
\label{sec:typesofpaths}

Let us start with an overview of needed concepts and definitions. 

\begin{definition}%[Lattice paths] 
A {\it step set} $\stepset \subset \Z$ is a finite set of integers. 
The elements of $\stepset$ are called \emph{steps}. 
A \emph{lattice path} $\walksym = (s_1,\ldots,s_n)$ is a sequence of steps, such that $s_j \in \stepset$ together with a starting point $\walk{0} \in \Z$. 
The \emph{length}~$|\walksym|$ of a lattice path is equal to the number of its steps. 
\end{definition}

Throughout this work we assume the starting point $\omega_0=0$. 
To avoid pathological cases we assume that the step set contains at least one negative and one positive number. 
Although the paths are one dimensional objects, it proves convenient to picture them as two dimensional directed paths.
To do so, each step $s_j$ is associated with the two dimensional jump $(1,s_j)$. 
This leads to the equivalent interpretation of the \emph{geometric realization} of a path by a sequence of vertical altitudes $(\walk{0},\walk{1},\ldots,\walk{n})$ where $\walk{i}-\walk{i-1} = s_i$ for $i=1,\ldots,n$.
Classical examples are \emph{Dyck paths} defined by the step set $\stepset = \{-1,1\}$ and \emph{Motzkin paths} defined by the step set $\stepset = \{-1,0,1\}$.  
For more details we refer to~\cite{motz48,dosh77,bern99}.

	Depending on the spatial constraints we consider the following types of lattice paths (see Table~\ref{tab:dirTypes}).
	
	\begin{definition}
	\label{def:pathtypes}
		A {\it bridge}\index{lattice path!bridge} is a path that ends on the $x$-axis, i.e.~$\walk{n}=0$. 
		A {\it meander}\index{lattice path!meander} is a path that always stays above the $x$-axis, i.e.~$\walk{i} \geq 0$.
		An {\it excursion}\index{lattice path!excursion} is a path that is at the same time a meander and a bridge. 
	\end{definition}
	
		Their generating functions have been fully characterized in~\cite{bafl02} by means of analytic combinatorial techniques, which lie also at the core of our approach. The standard reference for this body of techniques is the authoritative book by Flajolet and Sedgewick~\cite{flaj09}.

\renewcommand{\baselinestretch}{1.1}
\begin{table}[ht]
	\begin{center}
	\begin{tabular}{|c|c|c|}
		\hline                        & ending anywhere & ending at $0$ \\
		\hline 
		%& & \\ 
				                              & \multirow{3}{*}{\includegraphics[width=4cm]{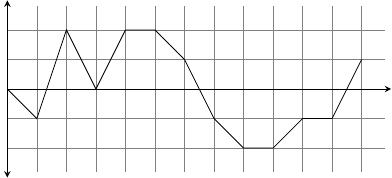}} & \multirow{3}{*}{\includegraphics[width=4cm]{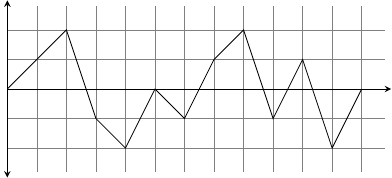}}\\
				                              %& & \\
				                              & & \\
				                unconstrained	& & \\
				                	(on $\Z$)	  & & \\
		                                  &  walk/path                   & bridge\\
		                                  &  $W(z) = \frac{1}{1-zP(1)}$          & $B(z) = z \sum\limits_{i=1}^c\frac{u_i'(z)}{u_i(z)}$\\ 
		                                  %& & \\				
	\hline 
	%& & \\ 
				                   & \multirow{3}{*}{\includegraphics[width=4cm]{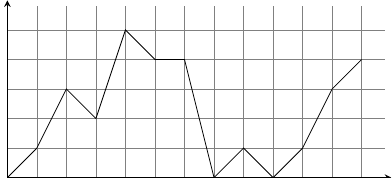}} & \multirow{3}{*}{\includegraphics[width=4cm]{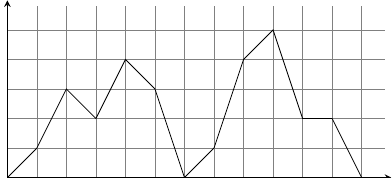}}\\
				                   %& & \\
				                   & & \\
				                constrained	& & \\
				                (on $\Z_{\geq 0}$)	& & \\
		                                &  meander                       & excursion\\
		                                &  $M(z) = \frac{1}{1-zP(1)}\prod\limits_{i=1}^c(1-u_i(z))$          & $E(z) = \frac{(-1)^{c-1}}{p_{-1}z}\prod\limits_{i=1}^c u_i(z)$\\
		                                %& & \\	
		\hline
\end{tabular}
\end{center}
\caption{The four types of paths: walks, bridges, meanders, and excursions; and the corresponding generating functions {\cite[Fig.~1]{bafl02}}.}
\label{tab:dirTypes}
\end{table}
\renewcommand{\baselinestretch}{1.0}

	\begin{remark}
		In the context of Motzkin paths we refer to Motzkin walks/meanders/bridges/excursions with the meaning of walks/meanders/bridges/excursions with Motzkin steps.
	In common literature Motzkin paths are often defined as Motzkin excursions, e.g.~in \cite{dosh77}, and Motzkin meanders as their prefixes. 
	\end{remark}
			
		For a given step set $\stepset$, we define the respective \emph{system of weights} 
	as $ \{p_s : s \in \stepset\}$ where $p_s >0$ is the associated weight with step $s \in \stepset$. 
		The \emph{weight of a path} is defined as the product of the weights of its individual steps. 
	For example, when all weights are equal to $1$, every path has weight $1$, which is precisely the unweighted model. 
	Conversely, if $\sum_{s \in \stepset} p_s = 1$ it is a probabilistic model of paths, i.e.,~step $s$ is chosen with probability $p_s$.
		
	The following definition is the algebraic link between weights and steps. 
			
	\pagebreak
	\begin{definition}%[step polynomial]
		The \emph{step polynomial} of $\stepset$ is defined as the polynomial in $u, u^{-1}$ (a Laurent polynomial)
		\begin{align*}
			P(u) := \sum_{j=1}^m p_j u^{s_j}.
		\end{align*}
		Define $c := - \min_j s_j$ and $d := \max_j s_j$ as the two extreme step sizes, and assume throughout $c,d >0$ to avoid trivial cases. The \emph{kernel equation} is defined by
		\begin{align}
			\label{eq:kerneleq}
			1-zP(u) &= 0, &&\text{ or equivalently } &
			u^c - z(u^c P(u)) &=0.
		\end{align}
		The quantity $K(z,u) := u^c - z u^c P(u)$ is called \emph{kernel}.
	\end{definition}
	
	The kernel plays a crucial \role in the asymptotic analysis of lattice paths. 
	The method behind this analysis is known as the \emph{kernel method}.
	For more details on its history the interested reader is referred to \cite[Chapter~1]{BanderierWallner16}. 
	At the heart of it lies the observation that the kernel equation is of degree $c+d$ in $u$, and therefore possesses generically $c+d$ roots. These correspond to branches of an algebraic curve given by the kernel equation. From the theory of algebraic curves and Newton-Puiseux series, for $z$ near $0$ one obtains $c$ ``small branches'' that we call $u_1(z), \ldots, u_c(z)$ and $d$ ``large branches'' $v_1(z), \ldots, v_d(z)$. For being well-defined, we restrict ourselves to the complex plane slit along the negative real axis. 
	
	They are called ``small branches'' because they satisfy $\lim_{z \to 0} u_i(z) = 0$, whereas the ``large branches'' satisfy $\lim_{z \to 0} |v_i(z)| = \infty$. 
	Banderier and Flajolet~\cite{bafl02} showed, that the generating functions of bridges, excursions and meanders can be expressed in terms of the small branches and the step polynomial; see Table~\ref{tab:dirTypes}.
	Furthermore, they derived that there is exactly one small and one large branch that is real and positive near $0$. 
	Let these two branches be $u_1(z)$ and $v_1(z)$, respectively, and call them \emph{principal branches}. 
	Their analytic properties are responsible for the asymptotic behavior of bridges, excursions, and meanders; compare \cite[Theorems 3 and 4]{bafl02}. 
	Let us recall the results and key definitions which are important for our exposition.
	
	The \emph{structural constant} $\tau$ is the unique positive real number $\tau>0$, such that $P'(\tau)=0$; see~\cite[Lemma~1]{bafl02}. 
	From it we define the \emph{structural radius} as 
	\begin{align*}
		\rho &:= \frac{1}{P(\tau)}.
	\end{align*}
		
	A walk is called \emph{periodic} with period $p$ if there exists a polynomial $H(u)$ and integers $b \in \Z$ and $p \in \N$, $p>1$ such that $P(u) = u^{b} H(u^p)$. Otherwise its called \emph{aperiodic}. 
	A systematic approach to deal with periodic cases was derived in~\cite{BanderierWallner16}.	
	Therefore the subsequent condition of aperiodicity imposes no significant restriction to the models.
	Then, for an aperiodic step set the principal branches $u_1(z)$ and $v_1(z)$ are analytic on $z \in (0,\rho)$ and they satisfy, respectively, the following singular expansion for $z \to \rho^-$, with $C := \sqrt{2 \frac{P(\tau)}{P''(\tau)}}$:
	\begin{align}
		\label{eq:u1asy}
		\tau \mp C \sqrt{1-\frac{z}{\rho}} + \LandauO\left(1 - \frac{z}{\rho} \right).
	\end{align}

	The previous result is a direct consequence of the implicit function theorem. 	
	We end this section with an extension of this result which we will use in the sequel to prove the conditions of Hypothesis~[H'].

\begin{proposition}
	\label{prop:decompu1}
	Let $u_1(z)$ and $v_1(z)$ be the principal small and large branches of the kernel equation $1-zP(u)=0$. Then there exists a neighborhood $\neigh$ of $\rho$ such that for $z \to \rho$ in $\neighrayrho$ they have a local representation of the kind
	\begin{align*}
		a(z) + b(z) \sqrt{1-z/\rho},
	\end{align*}
	where $a(z)$ and $b(z)$ are analytic functions for every point $z \in \neighrayrho$, $z \neq z_0$. We have $a(\rho)=\tau$, and $b(\rho) = -C$ for $u_1(z)$ or $b(\rho) = C$ for $v_1(z)$, respectively. 
	The other branches $u_2(z), \ldots, u_c(z)$ and $v_2(z), \ldots, v_d(z)$ are analytic in a neighborhood of $\rho$.
\end{proposition}

\begin{proof}
The branches $u(z)$, which we use as a shorthand for $u_i(z)$ and $v_i(z)$, are implicitly defined by the kernel equation~\eqref{eq:kerneleq}: $\Phi(z,u(z))=0$, where $\Phi(z,u)=1-zP(u)$. We will apply the singular implicit function theorem, in particular the version given as in \cite[Lemma VII.3]{flaj09}. Firstly, it is easy to check that $\Phi(z,u)$ satisfies the following conditions confirming the hypotheses: $\Phi(\rho,\tau)=0$, $\Phi_z(\rho,\tau) = -\rho^{-1} \neq 0$, $\Phi_u(\rho,\tau)= 0$, and $\Phi_{uu}(\rho,\tau) = -\rho P''(\tau) \neq 0$. Note that the last equation is not equal to $0$ because $P(u)$ is a convex function for real values of $u$. 
	
	This gives two possible solutions $y_1(z)$ and $y_2(z)$ which correspond to the principal small branch~$u_1(z)$ and the principal large branch~$v_1(z)$, respectively.
	These fulfill the claimed asymptotic expansions at the singularity $z=\rho$.
	A simple calculation shows that they give the claimed values for $a(\rho)$ and $b(\rho)$. 
	Thus, we recovered the asymptotic expansion \eqref{eq:u1asy}.
	Finally, the analytic nature of $a(z)$ and $b(z)$ follows from the Weierstrass preparation theorem; we refer to \cite{drmo94,drmo94b}, for an analytic presentation \cite[Theorem B.5]{flaj09}, or for an algebraic presentation \cite[Chapter 16]{abhy90}.
	
	It remains to discuss the analytic character of the other small branches.
	Here, the claimed result follows from the analytic version of the implicit function theorem: Consider $\widetilde{\Phi}(z,u) := \frac{\Phi(z,u)}{(u-u_1(z))(u-v_1(z))}$. Solving this function for $u$ gives the solutions of $\Phi(z,u)=0$ not equal to $u_1(z)$ or $v_1(z)$. But $\widetilde{\Phi}_u(\rho,\tau) \neq 0$ and therefore, these solutions are analytic in a neighborhood of $\rho$.
\end{proof}

\subsection{Returns to zero}
\label{sec:returns}

A \emph{return to zero} is a point of a walk of altitude~$0$ different from the starting point; see Figure~\ref{fig:signchanges}.  
The limit law for the number of returns to zero in excursions~\cite{bafl02} and bridges~\cite{bawa15b} behaves like a negative binomial distribution,
while in the more general reflection-absorption model~\cite{bawa15b} it behaves (after suitable rescaling) either like a normal, a Rayleigh, or a negative binomial distribution. 
In this section we consider in a unified manner the limit law of returns to zero in walks, and thereby generalize the results of Feller~\cite[Problems~$9$-$10$]{fell68} and Barton~\cite[Remark of Barton]{SkSh57}.

A standard technique is the decomposition into a sequence of ``minimal'' bridges (sometimes called arches). 
These are bridges that touch the $x$-axis only at the beginning and at the end. 
Then, the number of returns to zero corresponds to the number of these minimal bridges. 
Observe that every walk can be decomposed into a maximal initial bridge and a walk that never returns to the $x$-axis%
; see Figure~\ref{fig:retzerotail}. Let us denote the  generating function of the second part, which we call its \emph{tail}, by $T(z)$. From the above decomposition we have $T(z) = \frac{W(z)}{B(z)}$.

\begin{figure}[ht]
	\begin{center}	
		\includegraphics[width=0.7\textwidth]{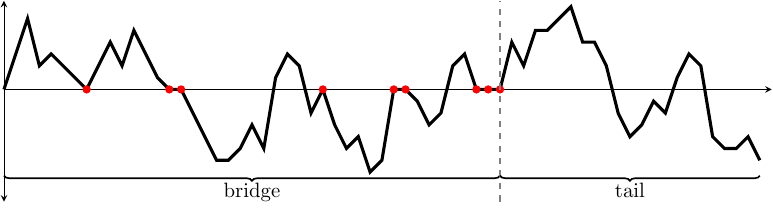}
		\caption{A walk with $9$ returns to zero decomposed into a bridge and a tail (a path that may cross but never returns to the $x$-axis). }
		\label{fig:retzerotail}
	\end{center}
\end{figure}

Let $W(z,u)$ be the bivariate generating function of walks of length marked by $z$ and returns to zero marked by $u$. 
Then, as the generating function of minimal bridges is $1-\frac{1}{B(z)}$, we get
\begin{align}
	\label{eq:returnsBGF}
	W(z,u) &= \frac{1}{1 - u \left(1-\frac{1}{B(z)}\right)} T(z) = \frac{W(z)}{u + (1-u)B(z)}.
\end{align}

Let $X_n$ be the random variable for the number of returns to zero of a random walk of length~$n$. Then, its law is given by
$$
	\PR(X_n = k) = \frac{[u^k z^n] W(z,u)}{[z^n] W(z,1)}. 
$$

\begin{theorem}[Limit law for returns to zero]
	\label{theo:mainRetZero}
	Let $X_n$ denote the number of returns to zero of a random, aperiodic walk of length $n$. Let $\delta=P'(1)$ be the drift. Then, for $n \to \infty$,
	\begin{enumerate} 
		\item if $\delta \neq 0$, we have convergence in law to a geometric distribution:
			\begin{align*}
				X_n \stackrel{d}{\to} \operatorname{Geom}\left(\frac{1}{B(1/P(1))}\right);
			\end{align*}
		\item if $\delta = 0$, we have convergence in law to a half-normal distribution:
			\begin{align*}
				\frac{X_n}{\sqrt{n}} \stackrel{d}{\to} \Hc\left(\sqrt{\frac{P(1)}{P''(1)}}\right).
			\end{align*}
	\end{enumerate}
\end{theorem}

\begin{proof}
	As a first step we compute the location of the dominant singularity of $W(z,u)$.
	We see that $[z^n] W(z,1) = [z^n]W(z) = P(1)^n$ and  we also directly get from its expression given in Table~\ref{tab:dirTypes} that $W(z)$ is singular at $\rho_1 := \frac{1}{P(1)}$. 
	Furthermore, due to the aperiodicity constraint $B(z)$ is only singular at $\rho$.
	We need to compare these two candidates.
	
	On the positive real axis the convex nature of $P(u)$ implies that $P(\tau)$ is its unique minimum. Hence, only two cases are possible: $\rho_1 < \rho$, if $\tau \neq 1$; or $\rho_1 = \rho$, if $\tau = 1$. 
	These two cases are also characterized by $\delta \neq 0$ or $\delta = 0$, respectively. 
	
	In the first case $W(z)$ is responsible for the dominant singularity. Then, we get (as $B(z)$ is analytic for $|z|<\rho$) for fixed $u \in (0,1)$
	\begin{align*}
		[z^n] W(z,u) &= \frac{1}{B\left(\rho_1\right)} \frac{P(1)^n}{1 - u\left(1-\frac{1}{B\left(\rho_1\right)}\right)} + \Landauo(P(1)^n).
	\end{align*}
	Thus, by the continuity theorem~\cite[p.~280]{fell68} the limit distribution is a geometric distribution with parameter $\frac{1}{B\left(\rho_1\right)}$.
	
	In the second case $\tau=1$ or $\delta=0$, we will apply Theorem \ref{theo:theo4}. 
	We know that $B(z) = z \sum_{j=1}^c \frac{u_j'(z)}{u_j(z)}$ is analytic for $|z| < \rho$; see \cite[Theorem~3]{bafl02}. Due to the aperiodicity, $\rho$ is the only singular point on the circle of convergence. Furthermore, $u_1(z)$ is the only small branch which is singular there, hence
	\begin{align}
		\label{eq:signBtau1}
		B(z) &= \frac{C}{2\tau \sqrt{1-z/\rho}} + \LandauO(1), \qquad C = \sqrt{2 \frac{P(\tau)}{P''(\tau)}},
	\end{align} 
	for $z \to \rho$. Then, Proposition~\ref{prop:decompu1} implies that $1/W(z,u)$ has a decomposition of the kind~\eqref{eq:decompFinv}.
	In particular, from \eqref{eq:signBtau1} we get that
	\begin{align*}
		\frac{1}{W(z,u)} &= \left(1-\frac{z}{\rho} \right)u + \frac{C}{2}(1-u) \sqrt{1-\frac{z}{\rho}} + \LandauO\left( \left(1-\frac{z}{\rho} \right) (1-u) \right), 
	\end{align*}
	for $z \to \rho$ and $u \to 1$, with $g(\rho,1)=h(\rho,1)=g_u(\rho,1)=g_{uu}(\rho,1)=0$; and $g_z(\rho,1)=-P(1)$ and $h_u(\rho,1) = - \sqrt{\frac{P(1)}{2 P''(1)}}$. 
	Finally, the analytic continuation follows from the algebraic nature of the involved generating functions; see e.g.~\cite[Lemma~VII.4]{flaj09}. 
	Note that $W(z,u)$ admits a pole $\rho(u)$ for $\Re(u)>1$ where $\rho(u)$ is analytic at $u=1$. 
	Furthermore, we have $|\rho(1+t i)|$ has a maximum for $t=0$, which implies that the domain can be extended beyond $\Re(u)=1$. 
	Hence, Theorem \ref{theo:theo4} yields the result.
\end{proof}

The derived results are summarized in Table~\ref{tab:compretsign} for the simple case of Motzkin paths with the step polynomial $P(u) = \frac{p_{-1}}{u} + p_0 + p_{1} u$. 
Note that in order to simplify the claims of the geometric cases we use the later derived result from Equation~\eqref{eq:motzkinBrho1}. 
All parameters are completely explicit in terms of the chosen weights $p_{-1},p_0$, and $p_1$.

This subsection nicely exemplified the simple applicability of our main theorem as a blackbox. 
What made the exposition simpler to follow was the additive nature of the parameter. 
In the next section we turn to a harder parameter of a different nature: the extremal parameter height.

\subsection{Height}
\label{sec:height}

For a path of length $n$ we define the \emph{height} as its maximally attained $y$-coordinate; see Figure~\ref{fig:motzkinheight}. Formally, let $\omega = (\omega_0,\omega_1,\ldots,\omega_n)$ be a walk of length $n$.
Then its height is given by
$$
	\max_{k \in \{0,\ldots,n\}} \omega_k.
$$

\begin{figure}[ht]
	\begin{center}	
		\includegraphics[width=0.7\textwidth]{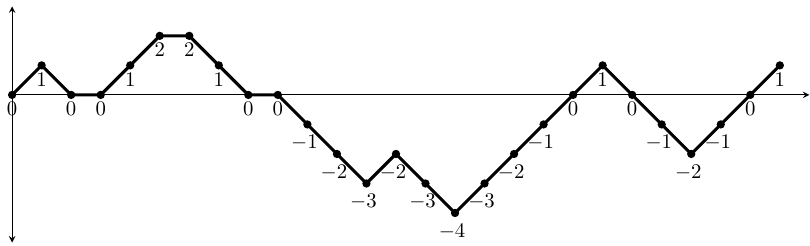}
		\caption{A Motzkin walk of height $2$. The relative heights are given at every node.}
		\label{fig:motzkinheight}
	\end{center}
\end{figure}

We define the bivariate generating function $F(z,u) = \sum_{n,h \geq 0} f_{nh}z^n u^h$. The coefficient~$f_{nh}$ is equal to the number of walks of height~$h$ among walks of length~$n$. 
Let $M(z,u) = \sum_{n,h \geq 0} m_{nh} z^n u^h$ be the generating function of meanders, where $m_{nh}$ is the number of meanders of length~$n$ ending at final altitude~$h$. Banderier and Flajolet derived in \cite[Theorem~2]{bafl02} its closed-form as
\begin{align}
	\label{eq:meandersbf}
	M(z,u) &= \frac{\prod_{j=1}^c (u-u_j(z))}{u^c (1-zP(u))} = - \frac{1}{p_d z} \prod_{\ell=1}^d \frac{1}{u-v_{\ell}(z)}.
\end{align}

The combinatorial key result of this section is the following theorem which connects the generating function of walks of length $n$ and height $h$ with the one of meanders of length $n$ and final altitude $h$.

\begin{theorem}
	\label{theo:WienerHopf}
	The bivariate generating function of walks (where $z$ marks the length, and $u$ marks the height) 
	is given by
	\begin{align}
		\label{eq:Fzuheight}
		F(z,u) &= \frac{W(z)M(z,u)}{M(z)} = - \frac{1}{p_d z} \left(  \prod_{j=1}^c \frac{1}{1-u_{j}(z)} \right) \left( \prod_{\ell=1}^d \frac{1}{u-v_{\ell}(z)} \right).
	\end{align}
\end{theorem}

\begin{proof}
	Banderier and Nicod{\`e}me derived in \cite[Theorem 2]{bani10} the generating function $F^{[-\infty, h]}(z)$ for walks which do not cross the wall $y=h$. 
	From this we directly get the generating function $F^{[h]}(z)$ for walks that have height exactly $h$. 
	In order to increase readability we omit the dependency on $z$ in the following formulas involving $v_i(z)$.
	We have
	\begin{align*}
		F^{[h]}(z) 
			= \sum_{i=1}^d \frac{v_i-1}{1-zP(1)} \left(\frac{1}{v_i}\right)^{h+1} \!\!\! \prod_{\substack{1\leq j \leq d \\ j \neq i}} \frac{1-v_{j}}{v_{i}-v_j}.
	\end{align*}
	
	Then marking the height by $u$ and summing over all possibilities gives 
	\begin{align*}
		F(z,u) &= \sum_{h\geq 0} u^h F^{[h]}(z) = 
			\frac{\prod_{j=1}^d (1-v_j)}{1-zP(1)} \sum_{i=1}^d \frac{1}{u-v_i} \prod_{\substack{1\leq j \leq d \\ j \neq i}} \frac{1}{v_{i}-v_j}.
	\end{align*}
	In the first factor we use the factorization of the kernel given by $1-zP(1)=-z p_{d} \prod_{j=1}^c (1-u_j) \prod_{\ell=1}^d (1-v_\ell)$; compare~\eqref{eq:kerneleq}. 
	The remaining sum is then simply the partial fraction decomposition of $\prod_{i=1}^d \frac{1}{u-v_i}$. Therefore this gives the second equality of~\eqref{eq:Fzuheight}. 

	Finally, by the second representation in~\eqref{eq:meandersbf} we discover the formula for $M(z,u)$ in~\eqref{eq:Fzuheight}.
	Then, using the representation of $M(z)$ from Table~\ref{tab:dirTypes} we see that the remaining part is equal to $W(z)/M(z)$ because $W(z) = 1/(1-zP(1))$.
\end{proof}

We want to remark that the relation given by the first equality in~\eqref{eq:Fzuheight} can be interpreted as an instance of the famous 
 Wiener--Hopf factorization 
in probability theory~\cite[Chapter~XII]{Feller1971Probability2}. 
In our context, this identity is obviously directly related to the kernel equation~\eqref{eq:kerneleq}. Its simple structure suggests a combinatorial interpretation, or even a direct combinatorial proof. In order to answer this question, we will analyze it in more detail now. Let us start with its factor $\frac{W(z)}{M(z)} = \prod_{j=1}^c \frac{1}{1-u_{j}(z)}$. 

First, we introduce a new useful dualism.
Let $\walksym = (\walk{0},\walk{1},\ldots,\walk{n})$ be a path with $\walk{0}=\walk{n}=0$.
A \emph{non-negative excursion} $\walksym$ is a ``traditional'' excursion, i.e., $\walk{i} \geq 0$, whereas a \emph{non-positive excursion} $\walksym$ is defined by $\walk{i} \leq 0$. 

\begin{remark}
	\label{rem:posnegexc}
	Mirroring the steps along the $x$-axis it follows that the number of non-negative excursions of length $n$ is equal to the number of non-positive excursions of length $n$. 
\end{remark}

In the same manner, we introduce the notion of \emph{non-positive meanders} for walks $\walksym = (\walk{0},\walk{1},\ldots,\walk{n})$ with $\walk{0}=0$ and $\walk{i} \leq 0$ and denote their generating function by $M_{\leq 0}(z)$. Furthermore, let \emph{negative meanders} be non-positive meanders that never return to the $x$-axis (but start at $0$), and denote their generating function by $M_{<0}(z)$.
This generating function is computed in the next proposition and the result implies that it equals $W(z)/M(z)$.

\begin{proposition}	
	\label{prop:negmeanders}
	The generating functions of negative meanders and non-positive meanders are given by
	\begin{align*}
		M_{<0}(z) &= \prod_{j=1}^c \frac{1}{1-u_{j}(z)}, \\
		M_{\leq 0}(z) &= E(z) M_{<0}(z) = \frac{(-1)^{c-1}}{p_{-c} z} \prod_{j=1}^c \frac{u_j(z)}{1-u_{j}(z)}.
	\end{align*}
\end{proposition}

\begin{proof}
	The key idea is that non-positive meanders are meanders after mirroring the coordinate system along the $x$-axis. By doing so, the step polynomial $P(u) = \sum_{i=-c}^d p_i u^i$ changes to the \emph{mirrored step polynomial}
	\begin{align*}
		\widetilde{P}(u) = \sum_{i=-d}^c p_{-i} u^i.
	\end{align*}
	The small branches $\tilde{u}_i(z)$, which satisfy $1-z\widetilde{P}(\tilde{u}_i(z)) = 0$ are given by
	$
		\tilde{u}_i(z) = \frac{1}{v_i(z)},
	$
	where $v_i(z)$ are the large branches of the original kernel equation $1-zP(u)=0$. Finally, by~\eqref{eq:meandersbf}  and because of $P(1)=\widetilde{P}(1)$ we get
	\begin{align*}
		M_{\leq 0}(z) &= \frac{\prod_{j=1}^d (1-\tilde{u}_j(z))}{1-zP(1)} 
				  = \frac{(-1)^{d-1}}{p_d z} \left( \prod_{j=1}^d \frac{1}{v_j(z)} \right) \prod_{j=1}^c \frac{1}{1-u_{j}(z)},
	\end{align*}
	due to the factorization of the kernel equation into linear factors of small and large branches. The first factor $\frac{(-1)^{d-1}}{p_d z} \left( \prod_{j=1}^d \frac{1}{v_j(z)} \right)$ is equal to the generating function of excursions $E(z)$.
	
	Finally, note that any non-positive meander can be uniquely decomposed into a maximal initial non-positive excursion and a  negative meander.
\end{proof}

The previous result translates the above mentioned Wiener--Hopf factorization into the world of generating functions: Decompose a walk of length $n$ into two parts by cutting it at its unique last maximum. The second part is, after shifting it down by its height, a negative meander counted by $M_{<0}(z)$. After rotating the first part by $180$ degrees it corresponds to a meander ending at the previous maximum which is counted by $M(z,u)$; compare~\eqref{eq:meandersbf}.
This gives an alternative proof of~\eqref{eq:Fzuheight}.

We now turn our attention back to the limit laws for the height of walks.
Let $X_n$ be the random variable for the height of a random walk of length $n$. Thus,
$$
	\PR(X_n = k) = \frac{[u^k z^n] F(z,u)}{[z^n] F(z,1)}. 
$$

The following theorem concludes this section with the governing limit laws for the height of walks. Note in particular the different rescaling factors in each case.
\begin{theorem}[Limit law for the height]
	\label{theo:height}
	Let $X_n$ denote the height of a walk of length $n$. Let $\delta=P'(1)$ be the drift, and $\rho_1 = 1/P(1)$. 
	\begin{enumerate} 
		\item For $\delta < 0$ the limit distribution is discrete and characterized in terms of the large branches:
			\begin{align*}
				\lim_{n \to \infty} \PR( X_n =k ) = [u^k] \omega(u), \qquad \text{ where } \omega(u) = \prod_{j=1}^d \frac{1-v_j(\rho_1)}{u - v_j(\rho_1)}.
			\end{align*}
		\item For $\delta = 0$ the standardized random variable converges to a half-normal distribution:
			\begin{align*}
				\frac{X_n}{\sqrt{n}} \stackrel{d}{\to} \Hc\left(\sqrt{\frac{P''(1)}{P(1)}}\right).
			\end{align*}
		\item For $\delta > 0$ the standardized random variable converges to a normal distribution:
			\begin{align*}
				\frac{X_n-\mu n}{\sigma \sqrt{n}} &\stackrel{d}{\to} \Nc\left( 0,1 \right), &
				\mu &= \frac{P'(1)}{P(1)}, &
				\sigma^2 &= 
					\frac{P''(1)}{P(1)} + \frac{P'(1)}{P(1)} - \left(\frac{P'(1)}{P(1)}\right)^2 \!\!.
			\end{align*}		
	\end{enumerate}
\end{theorem}

\begin{proof}
	From the structure of the generating function in~\eqref{eq:Fzuheight} it is obvious that the result strongly depends on the limit law of the final altitude of meanders. This was analyzed in~\cite[Theorem~6]{bafl02}. 

	First, let us consider $\delta<0$. In this case it proves convenient to consider the equivalent representation of~\eqref{eq:Fzuheight} derived from the definition of the small and large branches and shown in~\eqref{eq:meandersbf} (set $u=1$):
	\begin{align*}
		F(z,u) &= \frac{1}{1-zP(1)} \prod_{\ell=1}^d \frac{1-v_{\ell}(z)}{u-v_{\ell}(z)}.
	\end{align*} 
	In this case we know that $\tau>1$, implying $\rho>\rho_1$ (equivalently, $P(1)>P(\tau)$) and that the dominant singularity arises at $z=\rho_1$. The product of the large branches is analytic for $|z|<\rho$; see~\cite{bafl02}. Hence, by standard methods (see e.g.~\cite[Theorem~VI.12]{flaj09}) we get the asymptotic expansion:
		\begin{align*}
			[z^n] F(z,u) &= P(1)^n \prod_{\ell=1}^d \frac{1-v_{\ell}(\rho_1)}{u-v_{\ell}(\rho_1)} + \Landauo(P(1)^n).
		\end{align*}
		
		Second, for drift $\delta=0$, we have $\tau=1$. Thus, $P(\tau)=P(1)$ and the singularity arises at $\rho=\rho_1=1/P(1)$. This means that the singularities of the two factors coalesce and we can apply Theorem~\ref{theo:theo4}.  
		
		Let $\varepsilon > 0$. Then for $|z-\rho|<\varepsilon$,  $|u-1|<\varepsilon$, and $\arg(z-\rho) \neq 0$ we consider
		\begin{align*}
			\frac{1}{F(z,u)} &= -p_d z (1-u_1(z))(u-v_1(z)) 
				\underbrace{\left(\prod_{j=2}^c (1-u_j(z))\right)}_{=: \bar{U}_1(z)}
				\underbrace{\left(\prod_{\ell=2}^d (u-v_{\ell}(z))\right)}_{=: \bar{V}_1(z,u)}.
		\end{align*}
		The products $\bar{U}_1(z)$ and $\bar{V}_1(z,u)$ are analytic in a neighborhood of $z \sim \rho$. However, the branches $u_1(z)$ and $v_1(z)$ both possess a square-root singularity; compare~\eqref{eq:u1asy}. Thus, by Proposition~\ref{prop:decompu1} we have the desired decomposition~\eqref{eq:decompFinv}
	$$
		\frac{1}{F(z,u)} = g(z,u)+h(z,u) \sqrt{1-z/\rho},
	$$
	where $g(z,u)$ and $h(z,u)$ are analytic functions. In particular, the asymptotic expansion reads as follows
		\begin{align*}
			\frac{1}{F(z,u)} 
				&= \kappa \left( C(1-z/\rho) - (u-1)\sqrt{1-z/\rho} \right) \\
				& \qquad + \LandauO\left( (1-z/\rho)^{3/2}\right) + \LandauO\left( (u-1)(1-z/\rho)^{1/2} \right),
		\end{align*}
		where $\kappa$ is a non-zero constant.
		We immediately see that $g(\rho,1) = h(\rho,1) = g_u(\rho,1) = g_{uu}(\rho,1) = 0$, and that $g_z(\rho,1) = -\kappa C/\rho$, and $h_u(\rho,1) = -\kappa$. 
	Hence, Theorem \ref{theo:theo4} yields the result with the constant $\sigma = \sqrt{2}\frac{h_u(\rho,1)}{\rho g_z(\rho,1)} = \sqrt{\frac{P''(1)}{P(1)}}$.
	
	Finally, in the case $\delta >0$ the same reasoning as in \cite{bafl02} gives the result, as the perturbation by $M_{<0}(z)$ does not pose any problems. Yet, an alternative proof can be given via the perturbed supercritical sequence scheme \cite{BanderierWallner2017CatastrophesFull}.
\end{proof}

As in the previous section, the results simplify considerably for Motzkin paths; see Table~\ref{tab:compretsign}. 
In the next section we treat a less classical parameter in lattice path theory: the number of sign changes. 
The analysis becomes even more involved, which is why we deal only with Motzkin paths.

\subsection{Sign changes of Motzkin walks}
\label{sec:sign}

In the third application we deal with (weighted) Motzkin walks. Their step polynomial is given by
\begin{align*}
	P(u) = \frac{p_{-1}}{u} + p_0 + p_{1}u.
\end{align*}

We say that nodes which are strictly above the $x$-axis have a \emph{positive sign} denoted by ``$+$'', whereas nodes which are strictly below the $x$-axis have a \emph{negative sign} denoted by ``$-$'', and nodes on the $x$-axis are \emph{neutral} denoted by ``$0$''. 
This notion gives for every walk $\walksym = (\walksym_0,\dots,\walksym_n)$ a sequence of signs. In such a sequence a \emph{sign change} is defined by either the pattern~$+(0)-$ or~$-(0)+$, where~$(0)$ denotes a non-empty sequence of zeros; see Figure \ref{fig:signchanges}.

\begin{figure}[ht]
	\begin{center}	
		\includegraphics[width=0.7\textwidth]{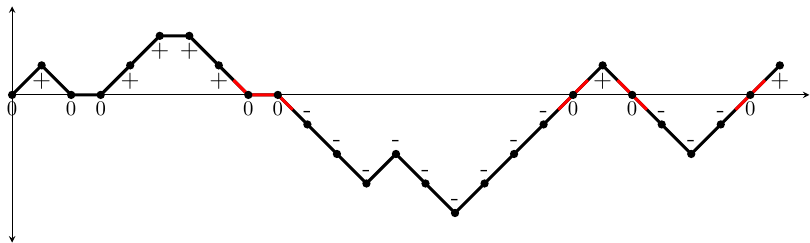}
		\caption{A Motzkin walk with $7$ returns to zero and $4$ sign changes. The positive, neutral, or negative signs of the walks are indicated by $+,0$, or $-$, respectively.}
		\label{fig:signchanges}
	\end{center}
\end{figure}

The main observation in this context is the non-emptiness of the sequence of zeros.
Geometrically this means that a walk has to touch the $x$-axis when passing through it.
Thus, we can count the number of sign changes by counting the number of maximal parts above and below the $x$-axis. 
The idea is to decompose a walk into an alternating sequence of non-negative and non-positive excursions ending with a non-negative or non-positive meander. 

We define the bivariate generating function 
$
	B(z,u) = \sum_{n,k \geq 0} b_{nk}z^n u^k,
$
where $b_{nk}$ denotes the number of bridges of size $n$ having $k$ sign changes. Furthermore, we define 
\vspace{-2mm}
\begin{align*}
	C(z) = \frac{1}{1-p_0 z},
\end{align*}
as the generating function of \emph{chains}, which are walks constructed solely from horizontal steps. Then, the generating function of (non-negative) excursions starting with a $+1$ step is 
\vspace{-1mm}
\begin{align*}
	%\label{eq:signeE1}
	E_1(z) &= \frac{E(z)}{C(z)}-1,
\end{align*}
since we need to exclude all excursions which start with a chain or are a chain. 
Obviously, this is also the generating function for non-positive excursions starting with a $-1$ step.

\begin{theorem}
	\label{theo:signbridgedecomp}
	The bivariate generating function of bridges 
	(where $z$ marks the length, and $u$ marks the number of sign changes) 
	is given by
	\begin{align}
		\label{eq:signBzu}
		B(z,u) &= C(z) \left( 1 + \frac{2 E_1(z)}{1-uE_1(z)} \right).
	\end{align}
\end{theorem}

\begin{proof}
	A bridge is either a chain, which has zero sign changes, or an alternating sequence of non-negative and non-positive excursions, starting with either of them. We decompose it uniquely into such excursions, by requiring that all except the first one start with a non-zero step. Therefore the first excursion is counted by $E(z)-C(z)$, whereas all others are counted by $E_1(z)$. The factor $2$ appears as the initial excursion could be non-negative or non-positive. The decomposition is shown in Figure~\ref{fig:signdecomp}.
\end{proof}

\begin{figure}[ht]
	\begin{center}	
		\includegraphics[width=0.8\textwidth]{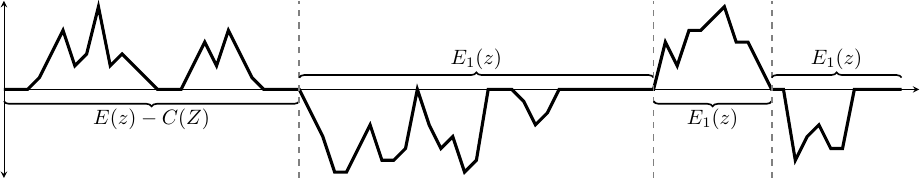}
		\caption{A bridge is an alternating sequence of non-negative and non-positive excursions. Here, it 
		starts with a non-negative excursion, followed by excursions starting with a non-zero step.}
		\label{fig:signdecomp}
	\end{center}
\end{figure}

We start our analysis by locating the dominant singularities of $B(z,u)$. 
First, we state some inherent structural results of the model which follow from direct computations. 

\begin{lemma}
	\label{lem:signMotzkinTauPtau}
	The structural constant $\tau$ which is the unique positive root of $P'(u)=0$ is $\tau = \sqrt{\frac{p_{-1}}{p_1}}$. 
	The structural radius is $\rho = \frac{1}{P(\tau)} = \frac{1}{p_0 + 2 \sqrt{p_{-1}p_1}}$.
\end{lemma}

Let $X_n$ be the random variable for the number of sign changes of a random bridge of length~$n$. 
Thus,
\[
	\PR(X_n = k) = \frac{[u^k z^n] B(z,u)}{[z^n] B(z,1)}. 
\]

\begin{theorem}[Limit law for sign changes of bridges]
	\label{theo:signbridges}
	Let $X_n$ denote the number of sign changes of a Motzkin bridge of length $n$. Then, for $n \to \infty$ the normalized random variable has a Rayleigh %
		%\footnote{The parameter $\lambda = \sigma^{-2}$ was used in \cite[Theorem 1]{DrSo97}.}
		limit distribution
	\begin{align*}
		\frac{X_n}{\sqrt{n}} \stackrel{d}{\to} \Rc\left(\sigma\right) \qquad \text{ and } \qquad
		\sigma = \frac{\tau}{2} \sqrt{\frac{P''(\tau)}{P(\tau)}},
	\end{align*}
	where $\tau = \sqrt{\frac{p_{-1}}{p_1}}$ and $\Rc(\sigma)$ has the density $\frac{x}{\sigma^2} \exp\left(-\frac{x^2}{2\sigma^2}\right)$ for $x \geq 0$. 
\end{theorem}

\begin{proof}%[Proof (Sketch)]
	We will apply the first limit theorem of Drmota and Soria, \cite[Theorem~1]{DrSo97}. (The conditions of Hypothesis [H] are the same as for Hypothesis [H'] with the additional requirement that $h(\rho,1) > 0$.) 
	
	Let us first analyze $B(z,1)$. Its dominant singularity is at $\rho$, as $1/p_0 > \rho = 1/(p_0 + 2 \sqrt{p_{-1}p_1})$; compare with Lemma~\ref{lem:signMotzkinTauPtau}.	
	Next we determine the decomposition at $z=\rho$ and $u=1$. From Proposition~\ref{prop:decompu1} it follows that $E_1(z)$ has a local representation of the kind
	\begin{align*}
		E_1(z) &= a_E(z) + b_E(z) \sqrt{1-z/\rho},
	\end{align*}
	where $a_E(z)$ and $b_E(z)$ are analytic functions around $z=\rho$ with $a_E(\rho)=1$ and $b_E(\rho)=-\frac{2}{\tau}\sqrt{\frac{2 P(\tau)}{P''(\tau)}}$. 
	From \eqref{eq:signBzu} we see that 
	\begin{align*}
		B(z,u) &= C(z) F(E_1(z),u), & \text{ where } &&
		F(y,u) &= 1+\frac{2y}{1-uy}.
	\end{align*}
	We can use the Taylor series expansion of
	\begin{align*}
		F(y,u)^{-1} &= \sum_{n, k \geq 0} f_{n k} (y-1)^n (u-1)^k,
	\end{align*}
	with $f_{00} = 0$ and $f_{10}=f_{01}=-1/2$ to show the desired decomposition:
	\begin{align*}
		B(z,u)^{-1} &= C(z)^{-1} F(E_1(z),u)^{-1} = g(z,u) + h(z,u) \sqrt{1-z/\rho}.
	\end{align*}	
	We have $g(\rho,1) = f_{00} = 0$, $h(\rho,1) = (1-\rho p_0)f_{10}b_E(\rho) 
	%= (1-\rho p_0) C/\tau 
	> 0$ and $g_u(\rho,1) = (1-\rho p_0) f_{01} 
	%= -(1-\rho p_0)/2 
	< 0$. 
	Finally, \cite[Theorem~1]{DrSo97} yields the result.
\end{proof}

Finally, we consider sign changes of walks. 
Since we want to count the number of sign changes we need to know whether the last non-zero node in a bridge has a positive or negative sign. 
Let \emph{positive bridges} be bridges whose last non-zero node has a  positive sign, and \emph{negative bridges} be bridges whose last non-zero node has a negative sign. Their generating functions are denoted by $B_+(z,u)$ and $B_-(z,u)$, respectively. Figure~\ref{fig:signdecomp} shows a negative bridge.

\begin{lemma}
	\label{lem:B+}
	The number of positive and negative bridges is the same and given by
	\begin{align*}
		B_+(z,u) = \frac{B(z,u)-C(z)}{2} = \frac{E(z)-C(z)}{1-uE_1(z)}.
	\end{align*}
\end{lemma}

\begin{proof}
	A bridge is a sequence of excursions, as shown in Figure~\ref{fig:signdecomp}. Therefore, mapping all positive excursions to negative ones, and vice versa, gives a bijection between positive and negative bridges. 
	For the formula, note that a bridge is either a positive bridge, a negative bridge, or a chain.
\end{proof}

In the following result we rediscover the tail $T(z) = W(z)/B(z)$, introduced in Section~\ref{sec:returns}. 
Here, $T(z)$ is the union of negative and positive meanders because we are dealing with Motzkin paths.
Then, $T(z)-1$ is the union of non-empty negative and positive meanders.

\begin{proposition}
	The bivariate generating function of walks $W(z,u) = \sum_{n,k \geq 0} w_{nk} z^n u^k$, where $w_{nk}$ is the number of all walks of length $n$ with $k$ sign changes, is given by
	\begin{align}
		\label{eq:Fzusignchanges}
		W(z,u) &= B(z,u) \frac{W(z)}{B(z)} + B_+(z,u) \left(\frac{W(z)}{B(z)} - 1 \right) (u-1),
	\end{align}
	where $W(z) = \frac{1}{1-zP(1)}$ is the generating function of walks.
\end{proposition}

\begin{proof}
	Combinatorially, a walk is either a bridge or a bridge concatenated by a positive or negative meander, i.e., a meander that does not return to the $x$-axis again. 
	In the second case an additional sign change happens if the last non-zero node in the bridge has a negative sign and continues with a non-empty positive meander, or vice versa. 
	By Lemma~\ref{lem:B+} $B_+(z,u)=B_-(z,u)$ and the desired form follows.
\end{proof}

The last ingredient for the proof of Theorem~\ref{theo:mainSignChanges} is the following (technical) lemma on the small branch~$u_1(z)$. 
%Remember from Lemma~\ref{lem:signMotzkinTauPtau} that $\tau^2 = p_{-1}/p_1$. 
It can also be used to simplify the results on the height from Theorem~\ref{theo:height} in the case of Motzkin walks because $u_1(z)v_1(z) = \frac{p_{-1}}{p_1}$; see Table~\ref{tab:compretsign}.
Note that these results generally do not apply to other models.

\begin{lemma}
	\label{lem:u1rho1}
	Let $P(u) = p_{-1}u^{-1} + p_0 + p_1 u$. Let $u_1(z)$ be the small branch of the kernel equation $1-zP(u)=0$ with $\lim_{z \to 0} u_1(z) = 0$, and define $\rho_1 := 1/P(1)$. Then 
	\begin{align*}
		%\label{eq:signu1rho1}
		u_1\left(\rho_1\right) &=
			\begin{cases}
				\mathrlap{1,}
				\hphantom{\tau^2 \left(\frac{P(1)}{P'(1)}\right)^3 \left( P(1)P''(1)-2P'(1)^2 \right),} & \text{ for } \delta < 0, \\
				\tau^2, & \text{ for } \delta >0,
			\end{cases}
		%&
		\\
		u_1'\left(\rho_1\right) &=
			\begin{cases}
				\mathrlap{-\frac{P(1)^2}{P'(1)},}
				\hphantom{\tau^2 \left(\frac{P(1)}{P'(1)}\right)^3 \left( P(1)P''(1)-2P'(1)^2 \right),} & \text{ for } \delta < 0, \\
				\tau^2 \frac{P(1)^2}{P'(1)}, & \text{ for } \delta >0,
			\end{cases}
		\\
	%\end{align*}
	%\begin{align*}
		u_1''\left(\rho_1\right) &=
			\begin{cases}
				-\left(\frac{P(1)}{P'(1)}\right)^3 \left( P(1)P''(1)-2P'(1)^2 \right), & \text{ for } \delta < 0, \\
				 \tau^2 \left(\frac{P(1)}{P'(1)}\right)^3 \left( P(1)P''(1)-2P'(1)^2 \right), & \text{ for } \delta >0.
			\end{cases}
	\end{align*}	
\end{lemma}

\begin{proof}
	For $\delta \neq 0$ we know that $u_1(z)$ is regular at $\rho_1$. 
	As by~\cite{bafl02} $u_1(z)$ is monotonically increasing on the positive real axis, we have $u_1(\rho_1) < u_1(\rho) = \tau = \sqrt{p_{-1}/p_1}$. Then, from the kernel equation $1-z P(u_1(z))=0$ for all $|z|<\rho$, we get the desired result. For the second and third claim one uses the implicit derivative of the kernel equation~\eqref{eq:kerneleq} and the previous results. 
\end{proof}

The next theorem concludes this discussion. Its proof is similar to the one of Theorem~\ref{theo:mainRetZero}. 

\begin{theorem}[Limit law for sign changes]
	\label{theo:mainSignChanges}
	Let $X_n$ denote the number of sign changes of Motzkin walks of length $n$. Let $\delta=P'(1)$ be the drift. 
	\begin{enumerate} 
		\item For $\delta \neq 0$ we get convergence to a geometric distribution:
			\begin{align*}
				X_n \stackrel{d}{\to} \operatorname{Geom}\left( \lambda \right), &&& 
				\text{ with } & 
				\lambda &= 
					\begin{cases}
						\frac{p_1}{p_{-1}}, & \text{ for } \delta < 0,\\
						\frac{p_{-1}}{p_{1}}, & \text{ for } \delta > 0.
					\end{cases}
			\end{align*}
		\item For $\delta = 0$ we get convergence to a half-normal distribution:
			\begin{align*}
				\frac{X_n}{\sqrt{n}} \stackrel{d}{\to} \Hc\left(\frac{1}{2} \sqrt{\frac{P''(1)}{P(1)}}\right).
			\end{align*}
	\end{enumerate}
\end{theorem}

\begin{proof}
	Let us start with an analysis of the dominant singularity. The most important term of~\eqref{eq:Fzusignchanges} decomposes into
	\begin{align*}
		\frac{W(z)}{B(z)} &= \frac{1}{1-zP(1)} \frac{u_1(z)}{z u_1'(z)}.
	\end{align*}
	The first factor is singular at $\rho_1=1/P(1)$ while the second one is singular at $\rho=1/P(\tau) \geq \rho_1$. 
	Thus, either $W(z)$ is singular or both are simultaneously.

	In the first case, $\delta \neq 0$, we use the coefficient asymptotics for the product of a singular and an analytic function \cite[Theorem VI.12]{flaj09}. 
	This method requires the evaluation of the analytic function $B(z)$ at the singularity $z=\rho_1$. 	
	Applying Lemma~\ref{lem:u1rho1} we compute
	\begin{align}
		\label{eq:motzkinBrho1}
		B\left(\rho_1\right) &=
			\begin{cases}
				-\frac{P(1)}{\delta}, & \text{ for } \delta < 0, \\
				\frac{P(1)}{\delta}, & \text{ for } \delta >0.
			\end{cases}
	\end{align}
	Then some additional calculations show for $\delta<0$ that
	\begin{align*}
		\PR(X_n = k) &= \frac{[u^k z^n] W(z,u)}{[z^n]W(z)} = [u^k]\left( - \frac{\delta}{p_{-1}}\right) \frac{1}{1-u\frac{p_1}{p_{-1}}} + \Landauo(1).
	\end{align*}
	This is a geometric distribution with parameter $\lambda = \frac{p_1}{p_{-1}}$. For $\delta>0$ the analogous result holds.

	In the second case, $\delta=0$, we additionally have $\tau=1$ and $\rho=\rho_1$. Then, we can apply Theorem~\ref{theo:theo4}.  
	A reasoning along the lines of Theorem~\ref{theo:signbridges} shows that
	$
		1/W(z,u) = g(z,u)+h(z,u) \sqrt{1-z/\rho},
	$
	where $g(z,u)$ and $h(z,u)$ are analytic functions. We omit the tedious calculations and directly derive the asymptotic form for $z \to \rho$.	
	For the tail we get by~\eqref{eq:signBtau1} the expansion 
	\begin{align*}
		\frac{W(z)}{B(z)} &= \frac{2}{C \sqrt{1 - z/\rho}} + \LandauO(1), \qquad C = \sqrt{2 \frac{P(1)}{P''(1)}},
	\end{align*}
	for $z \to \rho$. Thus, we have
	\begin{align*}
		\frac{1}{W(z,u)} &= \frac{2 C \rho p_{-1}}{\tau^2 (u-3)(u+1)} 
			\left( 
				\frac{4C}{\tau (u-3)} \left(1-\frac{z}{\rho}\right) + (u-1) \sqrt{1 - \frac{z}{\rho}}
			\right)   \\
			& \quad+ \LandauO\left( \left(1-\frac{z}{\rho} \right)^2\right) + \LandauO\left( \left(1-\frac{z}{\rho} \right) (1-u) \right), 
	\end{align*}
	for $|u-1| < \varepsilon$, $|z-\rho|< \varepsilon$ and $\arg(z-\rho) \neq 0$, with $g(\rho,1)=h(\rho,1)=g_u(\rho,1)=g_{uu}(\rho,1)=0$; and $g_z(\rho,1)=-\frac{C^2 p_{-1}}{\tau^3}$ and $h_u(\rho,1) = - \frac{C \rho p_{-1}}{2 \tau^2}$ together with the needed analytic continuation. 
	Hence, Theorem~\ref{theo:theo4} yields the result with the constant 
	$\sigma %= \sqrt{2}\frac{h_u(\rho,1)}{\rho g_z(\rho,1)} 
	  = \frac{1}{2} \sqrt{\frac{P''(1)}{P(1)}}$.	
\end{proof}	

For the problem of sign changes with a general step set, the same ideas are applicable, yet the computations are more tedious. The only new ingredient one would need is a variation of the kernel method, using both small and large roots presented in~\cite{BanderierWallner18Gascom}.

\subsection{Banach's matchbox problem} 
\label{sec:Banach}

As a last example let us consider Banach's matchbox problem~\cite[Chapter~VI.8(a)]{fell68}.
Consider a mathematician who always carries one matchbox in his right pocket and one matchbox in his left pocket. 
Each one contains exactly $n$ matches.
When he needs a match, he selects one pocket at random. 
At the moment when the first box is empty, how many matches remain in the other box?

We can use lattice paths to model this problem.
The parameter we will use is the difference between the fuller and the emptier box.
Thus, at the beginning we start at the origin. 
Every time we draw a match from the emptier box we draw an up step with weight $1/2$; every time we draw a match from the fuller box we draw a down step with weight $1/2$; if both boxes have the same size we draw an up step with weight $1$.
After $n$ up steps one of the boxes is empty, and the current altitude corresponds to the remaining matches in the other box.
Note that these paths are meanders which end with an up step.
Yet, their set of jumps are of two kinds: above the $x$-axis they are modeled by the step polynomial $P(u) = \frac{1}{2u} + \frac{u}{2}$ while on the $x$-axis by $P_0(u) = u$.
This is a certain instance of the reflection model studied by Banderier and the author in~\cite{bawa15b} (see also~\cite[Chapter~$4$]{Wallner16a}). 
In particular, it was shown, that in this case the final altitude of meanders follows a half-normal distribution; see~\cite[Table~$4$ of the arxiv version]{bawa15c}.
Chronologically, this was the starting point for the research of this paper.

For completeness, the generating function of Banach's matchbox problem is given by
\begin{align*}
	F(z,u) = \frac{zu}{ u(1-z) + (1-u)\sqrt{1-z}},
\end{align*}	
which is directly seen to satisfy all conditions of Theorem~\ref{theo:theo4}. 
Therefore, for $X_n$ being the number of remaining matches, we get
\begin{align*}
	\frac{X_n}{\sqrt{n}} \stackrel{d}{\to} \Hc\left(\sqrt{2}\right).
\end{align*}

Obviously, this problem can be generalized to arbitrary steps or different weights. The latter is known as Knuth's toilet paper problem~\cite{Knuth84Toilet}, which leads to a discrete distribution for negative drift, a half-normal distribution for zero drift (which is exactly the instance of Banach's matchbox problem), and a normal distribution for positive drift.

\section*{Acknowledgment}

I am greatly indebted to the referees for the detailed comments and  thorough work which significantly improved the structure and presentation of this work. 
Furthermore, I thank Cyril Banderier, Michael Drmota, Hsien-Kuei Hwang and Markus Kuba for interesting discussions on combinatorial schemes and limit laws. 
This research was partially supported by the Austrian Science
Fund (FWF) grants SFB~F50-03 and J~4162-N35.

\addcontentsline{toc}{section}{Bibliography} 
\bibliographystyle{plain}
\bibliography{literature_LP}

\end{document}